\documentclass[11pt]{article}
\usepackage[utf8]{inputenc}

\usepackage[title]{appendix}

\usepackage{url}
\usepackage{style}
\usepackage{authblk}
\author[1]{Vojt\v{e}ch Dvo\v{r}\'ak}
\author[2]{Ohad Klein}
\affil[1]{Department of Pure Maths and Mathematical Statistics, University of Cambridge, UK}
\affil[2]{Department of Mathematics, Bar Ilan University, Ramat Gan, Israel}
{
    \makeatletter
    \renewcommand\AB@affilsepx{: \protect\Affilfont}
    \makeatother

    \affil[ ]{Email addresses}

    \makeatletter
    \renewcommand\AB@affilsepx{, \protect\Affilfont}
    \makeatother

    \affil[1]{vd273@cam.ac.uk}
    \affil[2]{ohadkel@gmail.com}
}

\begin{document}


	
	\title{Probability Mass of Rademacher Sums\\ Beyond One Standard Deviation}



	\maketitle
	\begin{abstract}
		Let $a_1, \ldots, a_n \in \reals$ satisfy $\sum_i a_i^2 = 1$, and let $\varepsilon_1, \ldots, \varepsilon_n$ be independent uniformly random $\pm$ signs and $X = \sum_{i=1}^{n} a_i \varepsilon_i$. 
		It is conjectured that $X = \sum_{i=1}^{n} a_i \varepsilon_i$ has $\Pr[X \geq 1] \geq 7/64$. The best lower bound so far is $1/20$, due to Oleszkiewicz \cite{olesz96}. In this paper we improve this to $\Pr[X \geq 1] \geq 6/64$.
		\vspace{3mm}
		
	\end{abstract}
	
	
	\section{Introduction}
	
	\subsection{Background}
	
	Tail inequalities characterize the possible values of $\Pr[X \geq t]$ for various thresholds $t$ and random variables $X$ with mean $0$. We consider the case of Rademacher sums $X = \sum_{i \in [n]} a_i \varepsilon_i$ for real numbers $a_i$ and independently and uniformly distributed signs $\varepsilon_i \sim \{-1, 1\}$. We further focus on lower bounds to $\pr[X \geq t]$.
	
	If $t > \sqrt{\var(X)}$ we may have $\Pr[X \geq t] = 0$. If $t \leq 0$, clearly $\pr[X \geq t] \geq \frac{1}{2}$ because of the symmetry, and if $0<t < \sqrt{\var(X)}$, the Paley-Zygmund inequality gives $$\Pr[X \geq t] \geq \Pr[X > t] =\frac{1}{2} \pr[X^2 > t^2] \geq \frac{1}{2} (1-\frac{t^2}{\var(X)})^2 \frac{\var(X)^2}{\mathbb{E}[X^4]} > 0.$$
	
	What happens when $t = \sqrt{\var{X}}$? This case was studied in 1967 by Burkholder~\cite{Burk67} 
	with the conclusion that if $C_s = \inf_{X} \Pr[X \geq s\sqrt{\var(X)}]$, where the infimum is taken over all Rademacher sums, then $C_1 > 0$.
	It was then improved by Hitczenko and Kwapie\'n~\cite{HK94} to $C_1 \geq e^{-4}/8$, and then in 1996 by Oleszkiewicz~\cite{olesz96} to $C_1 \geq 1/20$. Hitczenko and Kwapie\'n~\cite{HK94} conjectured that $C_1 = 7/64$, having the tightness example $a_1 = \cdots = a_6 > 0$.
	
	We point out that this problem is a natural counterpart to the Tomaszewski's problem \cite{guy86}, which in the same setting of Rademacher sums, is concerned with the value of $\inf_{X} \Pr[|X| \leq \sqrt{\var(X)}]$. This problem attracted wide attention over the years 
	before it was finally settled recently by Keller and the second author \cite{KK20} -- the value is exactly $\frac{1}{2}$ (and henceforth, $C_{-1} = 3/4$).
	
	
	\subsection{Our results}
	
	The main result of our paper is the following.
	\begin{theorem}\label{thm:intro-olesz}
	    Any Rademacher sum $X = \sum_i a_i \varepsilon_i$ has
	    \[
	        \pr[X \geq \sqrt{\var(X)}] \geq 6/64.
	    \]
	\end{theorem}
	
	This theorem improves on the previously best known bound by Oleszkiewicz \cite{olesz96}, who derived an analogous result with the constant $\frac{1}{20}=0.05$ instead of our constant $\frac{6}{64}=0.09375$. We believe that our tools could be useful in order to prove the conjectured optimal bound of $\frac{7}{64}$. We make some progress toward this goal by handling certain difficult, near-extremal, classes of Rademacher sums. See further Section~\ref{difcases}. 
	
	While already $\pr[X > \sqrt{\var(X)}]$ might be $0$, as demonstrated by $X = 1\cdot \varepsilon_1$, the aforementioned proof by Oleszkiewicz~\cite{olesz96} in fact shows that $\pr[X > \sqrt{\var(X)}] \geq 1/20$ whenever $X$ is not of the form $a_i \varepsilon_i$. This bound is quite tight due to the example $a_1 = \cdots = a_4 > 0$ having $\pr[X > \sqrt{\var(X)}] = 1/16$. We show that this is indeed the extremal case.
    \begin{theorem}\label{thm:intro-olesz2}
        Any Rademacher sum $X = \sum_i a_i \varepsilon_i$ with $a_1, a_2 > 0$ has
        \[
            \Pr[X > \sqrt{\var(X)}] \geq 1/16.
        \]
    \end{theorem}
    Another inequality in this vein was conjectured by Lowther~\cite{lowther20} to be $C_{1/\sqrt{7}} = 1/4$, which is saturated by 
    $a_1 = \cdots = a_7 > 0$. We prove the following slightly weaker result.
    \begin{theorem}\label{thm:intro-lowther}
        Any Rademacher sum $X = \sum_i a_i \varepsilon_i$ has
        \[
            \pr[X > 0.35\sqrt{\var(X)}] \geq 1/4.
        \]
    \end{theorem}

    In the paper of Ben-Tal, Nemirovski and Roos~\cite{BN02}, the higher-dimensional analogue of the $C_1 = 7/64$ problem first appeared.
    In this setting, $X = \sum_i a_i \varepsilon_i$ with $a_i \in \reals^d$ and we are concerned with the probability $P(X) \defeq \pr\big[\norm{X}_2^2 \geq \be[\norm{X}_2^2]\big]$. The best result in this framework is due to Veraar~\cite{MV08} who showed that $P(X) \geq (\sqrt{12}-3)/15 \approx 0.031$. We remark that the following holds.
    \begin{theorem}\label{thm:intro-dim}
        Any $X = \sum_i a_i \varepsilon_i$ with $a_i \in \reals^d$ (for any $d \geq 1$) has
        \[
            \pr\li[\norm{X}_2^2 \geq \be[\norm{X}_2^2]\ri] \geq \frac{1-\sqrt{1-1/e^2}}{2} > 0.035.
        \]
    \end{theorem}
    
    Interestingly, we are not aware of any example that would demonstrate that the constant in Theorem \ref{thm:intro-dim} could not be as large as $\frac{7}{32}$ (which is the best one could hope for, since the result does not hold for any constant larger than that even when we only consider the case $d=1$, as commented previously).
    

	\subsection{Overview of techniques}
	
	A prevalent method for understanding the distribution of Rademacher sums is to partition their weights $\{a_i\}$ into two parts ($X=L + S$): large weights and small weights.
	Such partitioning is efficient, as the Rademacher sum having small weights is easy to analyze using quantitative versions of the Central Limit Theorem, while the Rademacher sum having large weights can be analyzed by enumeration over all the possibilities. In high level, this is the approach we take, but let us dive a little further into the details.
	
	Consider a Rademacher sum $X$ with $\var(X) = 1$. The problem addressed in Theorem~\ref{thm:intro-olesz} concerns with lower bounding $\Pr[X \geq 1]$.
	It turns out to be instructive to generalize this problem in two different ways:
	\begin{itemize}
	    \item Enable a more flexible threshold $t$, and not only $t=1$.
	    \item Impose a restriction on the weights: $|a_i| \leq a$ for a parameter $a \leq 1$.
	\end{itemize}
	Denote by $G(a, t)$ the answer to this more general problem: the infimum of $\pr[X \geq t]$, assuming $|a_i| \leq a$ ($a \in (0, 1]$, $t \in \reals$).
	Ultimately, Theorem~\ref{thm:intro-olesz} is encapsulated in the statement $G(1,1) \geq 6/64$, but we study $G(a,t)$ for all parameters $a,t$ at once. 
	
	The crucial point is that using the decomposition of our Rademacher sum to its large and small parts $X = L+S$, we can lower bound $G(a,t)$ by 
	\begin{equation}\label{eq:intro-recursion}
	    G(a,t) \geq \inf_L \be_{l\sim L}[G(a'/\sigma, (t-l)/\sigma)]
	\end{equation}
	where the infimum is taken over all possible values of $L$ induced by decompositions $X = L+S$ (for example, if we decompose $X=L+S$ with $L=a_1 \varepsilon_1 + a_2 \varepsilon_2$ whenever $a_1 + a_2 \geq 1$ and $L=a_1 \varepsilon_1$ otherwise, the infimum is taken over all $L=a_1 \varepsilon_1 + a_2 \varepsilon_2$ with $a_1 + a_2 \geq 1$ and with $a' = \min(a_2, \sqrt{1-a_1^2-a_2^2)}$ and $L=a_1 \varepsilon_1$ with $a' = \min(a_1, 1-a_1)$),
	the expectation is taken over $l$ being a realization of the random variable $L$, 
	$\sigma$ is the standard deviation of $S$ (that is, $\sqrt{1-\var(L)}$), and $a'$ is an upper bound on the weights of $S$ (whose value depends on the notion of how we decompose $X=L+S$).
	
	Equation~\eqref{eq:intro-recursion} enables one to recursively compute lower bounds on $G(a,t)$, and ultimately on $G(1,1)$. Roughly speaking, considering the decompositions $X=L+S$ with $L$ containing at most the three largest weights of $X$, we almost deduce Theorem~\ref{thm:intro-olesz}. However, using solely this method, we run into the following problem: In order to concretely define $G(a, t)$ through the recursive~\eqref{eq:intro-recursion}, we have to propose an initial lower estimate for $G(a, t)$. The initial estimate we use is `continuous' in nature (the Berry-Esseen inequality), and is unable to differentiate between bounds on $\Pr[X \geq t]$ and on $\Pr[X > t]$.
	However, there are various instances $X$, detailed in Section~\ref{difcases}, for which the stronger bound $\Pr[X > 1] \geq \frac{7}{64}$ (or even the bound $\Pr[X > 1] \geq \frac{6}{64}$, that we prove) does not hold! (e.g. the aforementioned $a_1=\cdots = a_4 > 0$.)
	
	To handle these more tight cases, 
	we take a completely different approach toward lower bounding $\Pr[X \geq 1]$ (i.e. Theorem~\ref{thm:intro-olesz}). That is, we \emph{upper bound} $\pr[X \in (-1,1)]$ (recall that $X$ is symmetric). To do that, we take the advantage of the following trade-off that usually arises. The collections $\lbrace a_1,\ldots,a_n \rbrace$ that either contain large mass of their variance in the small weights, or have their large weights very non-uniform, are harder to describe precisely, but are nevertheless easy to analyze, since usually stronger bounds hold for these. And  the collections $\lbrace a_1,\ldots,a_n \rbrace$ that contain only very small mass of their variance in the small weights and have their large weights quite uniform are easier to describe precisely, so despite only more tight bounds being true for these, we can derive those bounds.
	
	In various tight cases that arise, we commonly want to upper bound $\pr[X \in I]$ for some particular interval $I \subset \reals$. 
	To do that, we use a chain lemma, and a few related observations.
    
    In the chain lemma, we assume $X$ has some weights $a_1,\ldots,a_l$ which are `large' compared to the length of $I$ and consider the signed sums $\pm a_1 \pm \ldots \pm a_l$ -- ignoring the remaining `small' weights. We then associate the set of these $2^l$ signed sums with a hypercube graph in a natural way and then use a famous result of Erd\H{o}s~\cite{er45} to show that these sums are not very tightly concentrated. That in turn implies an upper bound on $\pr[X \in I]$. 
    
    
    Occasionally, we have to consider the case when $I$ is a very short interval (much smaller than $(-1,1)$). In such a case we divide the small weights into disjoint parts (a method introduced by Montgomery-Smith~\cite{SJ90}), so that each part has a substantial probability to be large compared to $I$, and apply the chain lemma 
    on these `large' parts to deduce that $\pr[X \in I]$ is small enough.

	\subsection{Difficult cases}\label{difcases}
	
	
	As described in the previous subsection, similarly to Tomaszewski's problem \cite{KK20}, the particular difficulty we are facing when trying to prove the conjecture $C_1=7/64$, are the cases when $\pr[X>1]< 7/64$ despite $\pr[X \geq 1] \geq 7/64$ (and their ‘neighborhoods', i.e. the collections with the few largest weights being roughly of the same sizes as in these cases). Notably, we have
	
	\begin{itemize}
	    \item for $a_1=1$, $\pr[X>1]=0$;
	    \item for $a_1=\ldots=a_4=\frac{1}{2}$, $\pr[X>1]=\frac{1}{16}$;
	    \item for $a_1=\ldots=a_9=\frac{1}{3}$, $\pr[X>1]=\frac{23}{256} \approx 0.0898\ldots<\frac{6}{64}$;
	    \item for $a_1=\frac{2}{3},a_2=\ldots=a_6=\frac{1}{3}$, $\pr[X>1]=\frac{6}{64}$;
	    \item for $a_1=a_2=\frac{1}{2},a_3=\ldots=a_{10}=\frac{1}{4}$, $\pr[X>1]=\frac{55}{512}<\frac{7}{64}$.
	\end{itemize}
	We have to deal with the first three cases even when proving our bound of $6/64$, and the last two cases are further hurdles on the way to the optimal bound.
	
	In our proof of the $6/64$ bound, big part of the argument is spent dealing with a subcase presented in Section~\ref{hardsubsection}, which corresponds to the collections `close to' the third case from above (which is the most intricate of the first three ‘barriers').
	
	In Section \ref{sec6}, we discuss these difficulties in more detail and make progress toward proving the $7/64$ bound, by proving it for families corresponding to the ‘neighbourhoods' of all the cases above except the third one.

	\subsection{Organization}
	In Section \ref{sec2}, we introduce notation, and define a certain type of a useful random process. In Section \ref{sec3}, we describe our main tools and prove Theorem~\ref{thm:intro-lowther}. We then use these tools in Section \ref{sec4} to prove Theorem \ref{thm:intro-olesz}, the main result of the paper. Section \ref{sec5} 
	contains the proof of Theorem \ref{thm:intro-olesz2}. In Section \ref{sec6}, we discuss the deficiency of our $6/64$ proof and propose how to advance toward $7/64$, proving the result in two out of three ‘difficult' cases. In Section \ref{sec7}, we discuss the high dimensional version of the problem as well as of the problem of Tomaszewski and prove Theorem \ref{thm:intro-dim}. Finally in Section \ref{sec8}, we summarize the open problems arising in the paper.
	
	Some of the more technical proofs from various parts of the paper are in Appendix~\ref{appA} and Appendix~\ref{fam2}.
	
	
	\section{Background and definitions}\label{sec2}
	
	In this section, we describe our setting, notation and assumptions that we are working with.
	
	Throughout, we will consider $X = \sum_{i=1}^{n} a_i \varepsilon_i $, where $\varepsilon_i$ are independent Rademacher random variables (i.e. independent random variables such that $\pr \big[ \varepsilon_i= +1 \big]=\pr \big[ \varepsilon_i= -1 \big]=\frac{1}{2}$) and $a_i$ are real numbers with $\sum_{i=1}^{n} a_i^2 = 1$. Moreover, we will always, without loss of generality, assume that $$a_1 \geq a_2 \geq \ldots \geq a_n>0 .$$ 
	
	Sometimes, we will work with variables $\lbrace b_i \rbrace$ or $\lbrace c_i \rbrace$ instead of $ \lbrace a_i \rbrace$. For these, we do not assume any conditions on their ordering unless so stated.
	
	At some points, we will also write $\mathbf{a}$ to denote $\lbrace a_1,\ldots,a_n \rbrace$.
	
	Our central aim will be to lower bound 
	
	\begin{equation}\label{trivi0}
	\pr \big[ X \geq 1 \big]=\frac{1}{2} \pr \big[ |X| \geq 1 \big].  
	\end{equation}

	At some points, we will work with $\pr \big[ X \geq 1 \big]$, while at other points, we will work with $\pr \big[ |X| \geq 1 \big]$. As expressed by \eqref{trivi0}, working with these two forms is of course equivalent and the entire proof could be rewritten using just one of these. We use both quantities in order to streamline the proof. 
	
	The function $D(a,x): (0,1] \times \mathbb{R} \to \mathbb{R}$ appears repeatedly throughout the proof. This is a particular function that we construct in subsection \ref{dynpro} and it has a property that for any $a \in (0,1], x \in \mathbb{R}$, if we have $a_1 \leq a$, then $\pr \big[ X \geq x \big] \geq D(a,x)$. While its computation is computer-aided, we emphasize that by writing `$D$', we always refer to its exact value, and not to its approximation.


	\section{Tools}\label{sec3}
		\subsection{Stopped random walks and chain argument}\label{stoprw}
		
	We start with an observation (following trivially from a well known result of Erd\H{o}s \cite{er45}) which we will use repeatedly.

    \begin{observation}\label{keyantichain}
		Let $b_1 \geq b_2 \geq \ldots \geq b_t >0$ be such that $b_{t-k+1}+\ldots+b_t \geq \alpha$ for some $\alpha > 0$ and $0< k \leq t$. Then, for any $x$ and any $b_{t+1},\ldots,b_s$, we have
		\[
		\pr  \big[ \sum_{i=1}^{s} b_i \varepsilon_i  \in \left( x - \alpha,x+\alpha \right) \big] \leq f(k,t) / 2^{t}
		\]
		where $f(k,t)$ denotes the sum of $k$ largest binomial coefficients of the form $\binom{t}{i}$ for some $i,\ 0 \leq i \leq t$.
	\end{observation}
	
	\begin{proof}
	If the probability was more than $f(k,t) / 2^{t}$ for some fixed $x$, then in particular we can choose signs $\varepsilon_{t+1}=\varepsilon_{t+1}',\ldots,\varepsilon_s=\varepsilon_s'$ in such a way that at least $f(k,t)+1 $ of the sums $$\pm b_1 \pm \ldots \pm b_t+b_{t+1} \varepsilon_{t+1}'+\ldots+b_s \varepsilon_s' $$
		are within less than $2 \alpha$ of each other. Let $$T= \lbrace \pm b_1 \pm \ldots \pm b_t+b_{t+1} \varepsilon_{t+1}'+\ldots+b_s \varepsilon_s' \rbrace.$$ Consider the bijection $g: T \rightarrow Q_t \simeq \lbrace \pm 1 \rbrace^t$ given by
		
		$$ b_1 \varepsilon_1 + \ldots +  b_t \varepsilon_t +b_{t+1} \varepsilon_{t+1}'+\ldots+b_s \varepsilon_s' \rightarrow (\varepsilon_1,\ldots,\varepsilon_t) .$$
		
		Let $S \subset T$ be the set of $f(k,t)+1$ elements of $T$ that are all within $2 \alpha$ of each other. Then by the result of Erd\H{o}s \cite[Theorem 5]{er45}, $g(S)$ contains an chain of length at least $k$. But that contradicts the assumption that $b_{t-k+1}+\ldots+b_t \geq \alpha$.
	\end{proof}
       
		

			Some times, we will only check the stronger condition that (in the cases $k=2,3$) no two out of the sums $x_0 \pm b_1 \pm \ldots \pm b_k$ are within less than $2 \delta$ of each other, which in particular implies no two hit any interval of the form $\left( x-\delta,x+\delta \right)$. For the special cases we need, we will use the following two straightforward observations to verify that.
		
		\begin{observation}\label{k2}
		Fix $\delta>0$ and $b_1,b_2 \geq \delta$ such that $|b_1-b_2| \geq \delta$. Then for any $x$ and any $b_3,\ldots,b_l$, we have $$\pr \big[\sum_{i=1}^{l} b_i \varepsilon_i  \in \left( x-\delta,x+\delta \right) \big] \leq \frac{1}{4}.$$
		\end{observation}
		
		\begin{proof}
		If the probability was more than $\frac{1}{4}$ for some fixed $x$, then in particular we can choose signs $\varepsilon_3=\varepsilon_3',\ldots,\varepsilon_l=\varepsilon_l'$ in such a way that at least two of the four sums $$\pm b_1 \pm b_2+b_3 \varepsilon_3'+\ldots+b_l \varepsilon_l' $$
		are within less than $2 \delta$ of each other. Looking at differences of this set, it can only happen if the set
		\[
		D=\lbrace b_1+b_2,\, b_1,\, b_2,\, |b_1-b_2|\rbrace
		\]
		contains some element smaller than $\delta$, and our assumptions guarantee that can not happen.
		\end{proof}
		
		\begin{observation}\label{k3}
			Fix $\delta>0$ and $c_1 \geq c_2 \geq c_3 \geq \delta$ such that $c_1-c_2,c_2-c_3 \geq \delta$, $|c_1-c_2-c_3| \geq \delta$. Then for any $x$ and any $c_4,\ldots,c_m$, we have $$\pr \big[\sum_{i=1}^{m} c_i \varepsilon_i \in \left( x-\delta,x+\delta \right) \big] \leq \frac{1}{8}.$$
		\end{observation}
		
			\begin{proof}
		If the probability was more than $\frac{1}{8}$ for some fixed $x$, then in particular we can choose signs $\varepsilon_4=\varepsilon_4',\ldots,\varepsilon_m=\varepsilon_m'$ in such a way that at least two of the eight sums $$\pm c_1 \pm c_2 \pm c_3 +c_4 \varepsilon_4'+\ldots+c_m \varepsilon_m' $$
		are within less than $2 \delta$ of each other. 
		Looking at differences of this set, it can only happen if the set $$D=\lbrace c_1,\, c_2,\, c_3,\,  c_1 \pm c_2,\,  c_1 \pm c_3,\,  c_2 \pm c_3,\,  c_1+c_2\pm c_3,\,  c_1-c_2+c_3,\, |c_1-c_2-c_3|\rbrace$$
		contains some element smaller than $\delta$; our assumptions guarantee it is impossible.
		\end{proof}
		
		In the easy cases, we are already given enough large weights as a part of our collection $\lbrace a_i \rbrace$ and can use these weights in the anti-concentration observations above. But if that is not true and we instead have a lot of very small weights, we can `generate' larger weights from them, as described in the subsection that follows.
		
			\subsection{The random process \tops{$W(S;x)$} and its success probability}
		
		For a set of real numbers $S=\lbrace d_1,\ldots,d_n \rbrace$ and a real number $x>0$, we denote by $W(S;x)$ (or by $W(d_1,\ldots,d_n;x)$) the following random process. We first fix a permutation $(i_1,\ldots,i_n)$ of $\lbrace 1,\ldots,n \rbrace$ which maximizes the probability that the process is successful (what it means for this process to be successful will be defined in due course). Next, we set $W_0=0$. After choosing $W_j$ for some $j<n$, if $|W_j| \geq x$, we set $$W_{j+1}=\ldots=W_n=W_j.$$ While if $|W_j|<x$, we let $\varepsilon_{i_{j+1}}$ be Rademacher random variable independent of the previous part of the process, and set $$W_{j+1}=W_j+d_{i_{j+1}} \varepsilon_{i_{j+1}} .$$

We denote by $r(S;x)$ (or by $r(d_1,\ldots,d_n;x)$) the final value of this process, i.e. $W_n$. We call it \textit{successful} if $|r(S;x)| \geq x$, and \textit{unsuccessful} otherwise.

We denote by $p(S;x)$ (or by $p(d_1,\ldots,d_n;x)$) the probability that the process is successful. In particular, if we have $|d_i| \geq x$ for any $i \in \lbrace 1,\ldots,n \rbrace$, clearly the corresponding process will always be successful because of our condition on ordering.

The following lemma is crucial for us when working with such random processes.
		
		
		
		\begin{lemma}\label{hittingprob}
		Assume we have positive reals $b_1,\ldots,b_k$ such that $\sum_{i=1}^{k} b_i^{2} \geq c \alpha^{2}$ for some fixed $c>1$ and fixed $\alpha>0$. Then $$p(b_1,\ldots,b_k;\alpha) \geq \frac{c-1}{c+3}.$$ Moreover, if for some $\eta \in (0,1)$, we have $b_1,\ldots,b_k \in (0, \eta \alpha ] \cup [\alpha, \infty)$, then $$p(b_1,\ldots,b_k;\alpha) \geq \frac{c-1}{c+\eta^{2}+2\eta}.$$
		\end{lemma}
		
		\begin{proof}
		If any term out of $b_1,\ldots,b_k$ has size at least $\alpha$, then clearly $p(b_1,\ldots,b_k;\alpha)=1$. So further assume none of the terms has size at least $\alpha$.
		
		Run the random process $W(b_1,\ldots,b_k; \alpha)$. Without loss of generality (and for notational convenience), we can assume that the ordering $b_1,\ldots,b_k$ maximizes the probability that the process is successful.
		We define the stopping time $T$ as follows. Let $T$ be the first time $i$ such that $|W_i| \geq \alpha$ if this time is at most $k$, and let $T=k$ otherwise. Let $p=p(b_1,\ldots,b_k;\alpha)$ be the probability that the process $W(b_1,..,b_k;\alpha)$ is successful, i.e. that it hits absolute value at least $\alpha$.
		
		Now we will lower and upper bound $\mathbb{E} \big[ W_T^{2} \big]$.
		
		Clearly $|W_T| \leq 2 \alpha$ (as every term has size at most $\alpha$ and $T$ is the first time we reach absolute value at least $\alpha$), and $|W_T| \leq \alpha$ in the case when we never hit absolute value at least $\alpha$. This gives
		
		\begin{equation}\label{upboundhitprob}
		\mathbb{E} \big[ W_T^{2} \big] \leq 4p\alpha^{2}+(1-p)\alpha^{2}.    
		\end{equation}
		But also, writing $A= b_1 \varepsilon_1 +\ldots+b_T \varepsilon_T $ and $B=b_{T+1} \varepsilon_{T+1} +\ldots+ b_k \varepsilon_k $ (setting $B=0$ if $T=k$), we collect the following easy observations. Firstly
		
		\begin{equation}\label{easyobs1}
		\mathbb{E}\big[ AB \big]=\sum_{T_0,x} \pr\big[ T=T_0,A=x \big] \mathbb{E} \big[ AB | T=T_0,A=x \big]=0,
		\end{equation}
		since for any $T_0,x$, we have $$\mathbb{E} \big[ AB | T=T_0,A=x \big] =x \mathbb{E}\big[ b_{T_0+1} \varepsilon_{T_0+1} +\ldots+b_k \varepsilon_k \big]=0.$$ Furthermore, noting that if $T=k$, then $B=0$, we obtain
		
		\begin{equation}\label{easyobs2}
		\mathbb{E}\big[ B^{2} \big] = \sum_{i=1}^{k-1} \pr\big[ T=i \big] \mathbb{E} \big[ B^{2} | T=i \big] = \sum_{i=1}^{k-1} \pr \big[ T=i \big]   \big( \sum\nolimits_{j=i+1}^{k} b_j^{2} \big) \leq p \sum_{i=1}^{k} b_i^{2}.
		\end{equation}
		Using \eqref{easyobs1} we conclude
		 
		 \begin{equation}\label{easyobs3}
		 \sum_{i=1}^{k} b_i^{2} =\mathbb{E} \big[(A+B)^{2} \big]=\mathbb{E}\big[ A^{2} \big]+\mathbb{E}\big[ B^{2} \big]+2\mathbb{E}\big[AB\big]=\mathbb{E} \big[A^{2} \big]+\mathbb{E} \big[ B^{2} \big].    
		 \end{equation}
		 Overall, combining \eqref{easyobs2} and \eqref{easyobs3} we conclude 
		 
		 \begin{equation}\label{easyobs4}
		 \mathbb{E} \big[ A^{2} \big] \geq (1-p)\sum_{i=1}^{k} b_i^{2} \geq (1-p)c \alpha^{2}.
		 \end{equation}
		Combining \eqref{upboundhitprob} and \eqref{easyobs4}, we obtain
		\[
		(1-p)c \alpha^{2} \leq \mathbb{E}\big[W_T^{2} \big] \leq 4p\alpha^{2}+(1-p)\alpha^{2}.
		\]
		Rearranging gives the first result.
		
		For the second result, just note that with our additional condition $b_1,\ldots,b_k \in (0, \eta \alpha ]$, we can replace the inequality $$\mathbb{E}\big[W_T^{2}\big] \leq 4p\alpha^{2}+(1-p)\alpha^{2}$$ by the stronger inequality $$\mathbb{E}\big[W_T^{2}\big] \leq p(1+\eta)^{2} \alpha^{2}+(1-p)\alpha^{2},$$ and conclude in exactly the same way as before.
		\end{proof}

    \subsection{Dynamic Programming bound}\label{dynpro}
    	Denote by $\wt{G}(a_1, x)$ the quantity $\inf_{X} \Pr[X > x]$ where the infimum is taken over all Rademacher sums $X$ with $\var(X) = 1$, and whose largest weight is at most $a_1$.
    
    	For the proof, it is useful to understand the function $\wt{G}$. Evaluating the function $\wt{G}(a_1, x)$ is in general harder than the problem we are concerned with in Theorem \ref{thm:intro-olesz}; the latter is, nonrigorously,
    	encapsulated in $\wt{G}(1, 1-\eps)$.
    
    	The goal of the dynamic-programming approach is to derive a lower bound on $\wt{G}$ by first obtaining some lower bound on $\wt{G}(a_1, x)$ for many values of $a_1,x$, and then using an iterative procedure to improve this bound further. The key tool enabling us to iterate is elimination of the largest weight (see Section~\ref{elimination} for more details about elimination).
    	
    	\subsubsection{Prawitz's smoothing Inequality}
    	
    	We will use a smoothing inequality of Prawitz \cite{pra72}. This inequality is a useful tool, providing bounds on the values of the cumulative distribution function of a random variable, in terms of a partial information regarding its characteristic function. Specifically, given the characteristic function of a random variable, it is possible to determine its distribution via the Gil-Pelaez formula. In the case of a Rademacher sum $X=\sum_i a_i \varepsilon_i$, we have the characteristic function $\varphi_X(t)=\prod_i \cos(a_i t)$. Assuming that we know the largest weight $a_1$, it is possible to estimate the value of $\varphi_X(t)$ for $t \ll 1/a_1$. Although for $t \gg 1/a_1$, we have no information regarding $\varphi_X(t)$, Prawitz' inequality is still capable of providing a decent estimate for the cumulative distribution function of $X$.
    	
    	While the inequality is applicable to all random variables, it was shown in~\cite{KK20} that its specialization to Rademacher sums gives tighter estimates.
    	
    		Prawitz' bound gives a lower bound on $\wt{G}(a_1, x)$, for all parameters $q \in [0,1],\ T > 0$:
    		\begin{equation}\label{eq:prawitz-eq}
    			\forall q \in [0,1],\, T > 0 \cc\qquad
    			\wt{G}(a_1, x) \geq F(a_1, x, T, q).
    		\end{equation}
    		Specifically, a formula for $F$ may be derived from~\cite[Proposition 4.2]{KK20} (which is derived from~\cite{pra72}):
    		
    		\begin{align}\label{eq:our-prawitz}
    		\begin{split}
    			F(a,x,T,q) = 1/2
    			&- \int_{0}^{q} \left|k(u,x,T)\right| g(Tu, a) \dd{u}
    			 - \int_{q}^{1} \left|k(u,x,T)\right| h(Tu, a) \dd{u} \\
    			&- \int_{0}^{q} k(u,x,T) \exp(-(Tu)^2/2) \dd{u},
    		\end{split}
    		\end{align}
    		where\footnote{$k$ can be smoothly continued to the range $u\in \{0, 1\}$ by setting $k(0,x,T)=1+Tx/\pi$ and $k(1,x,T)=0$.}
    		$k(u,x,T) = \frac{(1-u)\sin(\pi u + T u x)}{\sin(\pi u)} + \frac{\sin(T u x)}{\pi}$,
    		\[
    			g(v, a) =
    			\begin{cases}
    			\exp(-v^{2}/2) - \cos(a v) ^ {1/a^2}, & a v \leq \frac{\pi}{2}\\
    			\exp(-v^{2}/2)+1, & \mrm{otherwise}
    			\end{cases}
    			,
    			\quad h(v, a) =
    			\begin{cases}
    			\exp(-v^{2}/2),					& a v \leq \theta\\
    			(-\cos(a v))^{1/a^2},	    & \theta \leq a v \leq \pi\\
    			1,								& \mrm{otherwise}
    			\end{cases},
    		\]
    		$Z \sim N(0,1)$ is a standard Gaussian and $\theta=1.778 \pm 10^{-4}$ is the unique solution of $ \exp(-\theta^2/2) = -\cos(\theta)$ in the interval $[0,\pi]$. We note that $F(a,x,T,q)$ is a function (weakly) decreasing in $a$.

	\subsubsection{Recursion}
		Note that as in~\eqref{eq:intro-recursion}, by considering the two values that the sign of the largest weight can take (see subsection \ref{elimination} for more details), we have
		\begin{equation}\label{eq:DP-eq}
			\wt{G}(a_1, x) \geq \frac{1}{2} \inf_{a \in (0, a_1]} \li( \wt{G}\li(\frac{a}{\sqrt{1-a^2}}, \frac{x-a}{\sqrt{1-a^2}}\ri) + \wt{G}\li(\frac{a}{\sqrt{1-a^2}}, \frac{x+a}{\sqrt{1-a^2}} \ri) \ri).
		\end{equation}
		Hence, $\wt{G}$ is lower bounded by the lowest function satisfying both inequalities~\eqref{eq:prawitz-eq},~\eqref{eq:DP-eq}. Computationally, to obtain a concrete lower bound on $\wt{G}$, we iteratively define the functions
		\[
		\func{D_i}{(0,1)\times \reals}{\reals}
		\]
		by $D_0(a_1, x) = \max(F(a_1, x),\  \one\{x<0\}/2)$ with $F(a_1, x) = \sup_{T,q} \{ F(a_1, x, T, q) \}$ and
		\begin{equation}\label{eq:Dip1}
			D_{i+1}(a_1, x) = \max\Big(D_i(a_1, x), \frac{1}{2} \inf_{a \in (0, a_1]} \Big( D_i\Big(\frac{a}{\sqrt{1-a^2}}, \frac{x-a}{\sqrt{1-a^2}}\Big) + D_i\Big(\frac{a}{\sqrt{1-a^2}}, \frac{x+a}{\sqrt{1-a^2}} \Big) \Big) \Big),
		\end{equation}
		and observe that $\wt{G}(a_1, x) \geq D_i(a_1, x)$ for all $i$. Choosing a large $I$ ($I=10$ suffices) and writing
		\[
		D(a_1, x) = D_I(a_1, x)
		\]
		we derive
		\begin{equation}\label{eq:DD}
		    \forall X \in \mathcal{X}\cc (X= \sum b_i \varepsilon_i \andd |b_i| \leq a_1) \implies \Pr[X > x] \geq D(a_1, x).
		\end{equation}
		
		Note that $D$ is a function depending on two continuous variables, which cannot be stored programmatically. We compute $D_i(a_1, x)$ for $a_1 \in [0,1]$ and $x \in [-3,3]$ with granularity of $\delta=1/400$ ($a_1$ starting from $0$ and $x$ starting from $-3$). Correspondingly, we replace~\eqref{eq:Dip1} with a variant that feeds $D_{i+1}$ with arguments rounded up (to a multiple of $\delta$), hence underestimating $D_{i+1}$; This enables considering a finite set of $a \in [0, a_1]$ in the infimum at~\eqref{eq:Dip1}. We apply this rounding-up to both the $\frac{a}{\sqrt{1-a^2}}$ and the $\frac{x\pm a}{\sqrt{1-a^2}}$ arguments. Moreover, in any computation of $D(a_1, x)$ we round the arguments up to multiples of $\delta$. When $x < -3$ we round $x$ to $-3$, and when $x \geq 3$ we round $x$ to $\infty$ and set $D(a_1, \infty)=0$. 
		This results in a dynamic-programming method for computing $D_i(a_1, x)$.
		
		Our implementation of this computation can be found at \cite{us}.
 		
 		\paragraph{Several concrete values.}
 		Along the paper, we use the following lower bounds for values of $D$, derived by the described computation. 
 		\begin{equation}\label{eq:stash-of-D}
 		\begin{aligned}
 		    & D(0.35, 0.35) > \frac{1}{4},\qquad & D(0.3, 1) > \frac{3}{32},\\
 		    & D(0.3/\sqrt{0.51}, 0.3/\sqrt{0.51}) > \frac{3}{16}, \qquad & D(0.4,1) > \frac{1}{12}, \\
 		    &    D(0.5,0.5)>\frac{1}{6}, \qquad & D(0.34,1.42)>0.04,
 		    \\
 		    & D(0.43,1.42)>0.03, \qquad & D(0.51,1.01)=\frac{1}{16}.
 		\end{aligned}
 		\end{equation}
 		
 		
 		Note that $D(0.51,1.01)=1/16$ is a precise value (unlike the other values mentioned for which we just have lower bounds). On the one hand, we clearly see that $D(0.51,1.01) \leq 1/16$, as saturated by the weights $a_1=\ldots=a_4=1/2$. On the other hand, to derive $D(0.51,1.01) \geq 1/16$, it is crucial that we set $D_0(a_1, x) = \max(F(a_1, x), \one\{x < 0\} / 2)$ instead of just using $F(a_1,x)$. Our iterative procedure and the lower bounds on $F(a_1,x)$ are then enough to prove $D(0.51,1.01) \geq 1/16$.
 		
 	    \paragraph{Precision.} As described, the lower bound $D(a,x)$ we numerically get for $\wt{G}(a,x)$ is precise.
 		The only detail disregarded so far is the computation of $F(a_1, x)$. Programmatically we replace $F(a_1, x)$ by $F(a_1, x, \pi / a_1, 0.5)$, that is, we do not compute the maximum of $F(a_1, x, T, q)$ over all values of $T,q$, but set $T=\pi / a_1$ and $q=0.5$. Since we use $F(a_1, x)$ as a lower bound, this underestimation of $F(a_1, x)$ is valid. We further note that this choice of $T,q$ simplifies the first integrand in $F(a_1, x, T, q)$ to be continuous (specifically, $g(v,a)$ is applied only when $av\leq \pi / 2$). Finally, to numerically estimate the integrals appearing in the definition of $F(a,x,T,q)$ we take two approaches.
 		
 		In the first approach we compute the integrals appearing in~\eqref{eq:our-prawitz} verbatim by using the standard Python integrator \texttt{scipy.integrate.quad}, and check that the integrator estimates that its error is well below some constant ($0.01$) that we discount from $F(a,x,T,q)$. We also split the domains of integration so that the integrands are smooth in each subdomain. This evaluation of $F$ is simple, but requires relying on the accuracy of \texttt{scipy.integrate.quad}.
 		
 		In the second approach we compute the integrals with the trapezoid rule, using explicit bounds $B$ on the derivatives of the integrands (more accurately, we use that these are $B$-lipschitz functions), to get an explicit estimation of the integrals, together with a provable error estimates. The bounds $B$ are computed in~\cite[Appendix B.2]{KK20}.
 		
 		While the first approach is neat and simple, the second approach is transparent and reviewable. The accompanied code is available at~\cite{us}.
 		
		\subsection{Elimination}\label{elimination}
		Elimination is the process of replacing a probabilistic inequality in $X = \sum_{i=1}^{n} a_i \varepsilon_i$, by an inequality involving $Z = \sum_{i=m}^{n} a_i \varepsilon_i$ with $m > 1$. For example, the inequality
		\[
		    \Pr[X \geq 1] \geq 3/32
		\]
		is equivalent to the following inequality, which involves $Z = \sum_{i=2}^{n} a_i \epsilon_i$ (i.e. $m = 2$),
		\[
		    \Pr[Z \geq 1 - a_1] + \Pr[Z \geq 1 + a_1] \geq 3/16.
		\]
		via the law of total probability. A more elaborate derivation can be found at~\cite[Lemma~2.1]{KK20}.
		
	\subsection{A \tops{$1/\sqrt{7}$}-type inequality}\label{ssec:sqrt7}
	    Lowther ~\cite{lowther20} conjectured that $\pr[|X| \geq 1/\sqrt{7}] \geq 1/2$ is true for all Rademacher sums $X$ with $\var(X) = 1$.
	    In the proof of Theorem~\ref{thm:intro-olesz} we make use of Theorem~\ref{thm:intro-lowther}, i.e. $\Pr[|X| > 0.35] \geq 1/2$, which we henceforth prove.
	    
	    We split into two cases. If $a_1 > 0.35$, and $\varepsilon' = (-\varepsilon_1, \varepsilon_2, \ldots, \varepsilon_n)$, then at least one of $X(\varepsilon)$ and $X(\varepsilon')$ has absolute value more than $ a_1$, hence $\Pr[|X| > 0.35] \geq 1/2$.
	    If $a_1 \leq 0.35$, then we conclude using \eqref{eq:stash-of-D} since
	    \[
	        D(0.35, 0.35) > 1/4.
	    \]


\section{Proof of \tops{$\Pr[X \geq 1] \geq 3/32$}}\label{sec4}
	In this section we show that for any Rademacher sum $X$ with $\var(X) = 1$,
	\begin{equation}\label{eq:3/32}
	    \Pr[X \geq 1] \geq 3/32,
	\end{equation}
	that is, Theorem~\ref{thm:intro-olesz}. The proof splits into two main cases - the case when $a_1+a_2+a_3 \leq 1$ and the case when $a_1+a_2+a_3>1$.
	
	In the case $a_1+a_2+a_3 \leq 1$, the tools we have developed in subsections \ref{dynpro} and \ref{elimination} enable us to handle most of the subcases. Nevertheless, as discussed before, one can not hope for these tools to work in the subcase $(a_1,a_2,a_3) \approx (\frac{1}{3},\frac{1}{3},\frac{1}{3})$ and $a_1+a_2+a_3 \leq 1$. Thus, we spend majority of this subsection dealing with the subcase $a_3 \geq 0.325$ and $a_1+a_2+a_3 \leq 1$. To do that, we use the tools developed in subsection \ref{stoprw}. Our strategy is to show that the family of such collections $ \lbrace a_i \rbrace$ with $a_3 \geq 0.325$ and $a_1+a_2+a_3 \leq 1$ is contained in the union of several subfamilies, for each of which we can obtain the desired bound.
	
	In the case $a_1+a_2+a_3>1$, the proof is less lengthy. We divide it into several subcases and use the tools from subsections \ref{dynpro} and \ref{elimination} and crucially also Theorem \ref{thm:intro-lowther}, to resolve these cases.


	\subsection{Case \tops{$a_1+a_2+a_3 \leq 1$}}
		\subsubsection{Subcase \tops{$a_1 \leq 0.3$}}\label{ssec:0.31}
			Using \eqref{eq:stash-of-D} we have 
			\[
			D(0.3, 1) > 3/32
			\]
			implying the assertion~\eqref{eq:3/32} through~\eqref{eq:DD}.

        \subsubsection{Subcase \tops{$a_1 \geq 0.7$}}\label{ssec:0.72}
        Using elimination, in order to deduce~\eqref{eq:3/32} regarding $X=\sum_{i=1}^{n} a_i \varepsilon_i$ it suffices to check
		    \begin{equation}\label{eq:3}
		        \Pr\li[X' \geq \frac{1-a_1}{\sqrt{1-a_1^2}}\ri] \geq 3/16
		    \end{equation}
		    with $X' = \frac{1}{\sqrt{1-a_1^2}}\sum_{i=2}^{n} a_i \varepsilon_i$ the $a_1$-eliminated version of $X$. Using \eqref{eq:stash-of-D} we deduce~\eqref{eq:3} from 
		    \[
		        D(0.3/\sqrt{0.51}, 0.3/\sqrt{0.51}) > 3/16,
		    \]
		    since $a_1 + a_2 \leq 1$. This argument does not rely on $a_1+a_2+a_3 \leq 1$, but only assumes $a_1+a_2\leq 1$ (and $a_1 \geq 0.7$). This is used in Section~\ref{ssec:0.7}.

		\subsubsection{Subcase \tops{$a_3 \leq 0.325$} and \tops{$a_1 \in [0.3, 0.7]$}}
			Under the conditions $a_1 \geq 0.3$ and $a_3 \leq 0.325$ (and $a_2 \in [a_3, a_1]$), denote $\sigma_2 = \sqrt{1-a_1^2-a_2^2}$, and note that $a = \min(1-a_1-a_2, a_2, 0.325)$ is an upper bound on $a_3$. 
			We show in Appendix~\ref{app:0.325} that
			\begin{equation}\label{eq:0.325}
				\be_{\varepsilon \in \spm^2} \li[ D\li( \frac{a}{\sigma_2}, \frac{1 + a_1 \varepsilon_1 + a_2 \varepsilon_2}{\sigma_2}\ri) \ri] \geq 3/32,
			\end{equation}
			verifying~\eqref{eq:3/32} in this case, via elimination of $a_1, a_2$.
			
		\subsubsection{Subcase \tops{$a_3 \geq 0.325$}}\label{hardsubsection}
		    Let $Y = \sum_{i=4}^{n} a_i \varepsilon_i$ and denote:
			\begin{equation*}
			\begin{gathered}
			q_1=\pr\big[|Y| \geq 1-a_1-a_2-a_3\big], \qquad
			q_2=\pr\big[|Y| \geq 1-a_1-a_2+a_3\big], \\
			q_3=\pr\big[|Y| \geq 1-a_1+a_2-a_3\big], \qquad
			q_4=\pr\big[|Y| \geq 1+a_1-a_2-a_3\big], \\
			q_5=\pr\big[|Y| \geq 1-a_1+a_2+a_3\big], \qquad
			q_6=\pr\big[|Y| \geq 1+a_1-a_2+a_3\big], \\
			q_7=\pr\big[|Y| \geq 1+a_1+a_2-a_3\big], \qquad
			q_8=\pr\big[|Y| \geq 1+a_1+a_2+a_3\big].
			\end{gathered}
			\end{equation*}
			Then using elimination, we have $\pr[X \geq 1]=\frac{1}{16}(q_1+\ldots+q_8)$. Hence we are required to show
			\begin{equation}\label{qcondition}
			q_1 + \ldots + q_8 \geq 3/2.    
			\end{equation}

			The key lemma which lets us handle this case, 
			is the following.
			
			\begin{lemma}\label{intothreeclasses}
			Let $\mathcal{A}$ be the family of the collections $\mathbf{a}=(a_1,\ldots,a_n)$ with $n \geq 4$, $a_1 \geq \ldots \geq a_n>0$, $\sum_{i=1}^n a_i^2=1$, $a_1+a_2+a_3 \leq 1$ and $a_3 \geq 0.325$. Then $\mathcal{A}=\mathcal{A}_1 \cup \mathcal{A}_2  \cup \mathcal{A}_3$, where $\mathcal{A}_1,\mathcal{A}_2,\mathcal{A}_3$ are the subsets of $\mathcal{A}$ characterized by the following additional conditions:
			
			\begin{itemize}
			    \item $\mathcal{A}_1$: $a_4 \leq 7/40$,
			    \item $\mathcal{A}_2$: $q_1 \geq \frac{793}{1024}$,
			    \item  $\mathcal{A}_3$: $q_2,q_3 \geq \frac{37}{128}$.
			\end{itemize}
			\end{lemma}
			
			\begin{proof}
			Firstly, if we had $a_1+a_2+a_3=1$, then clearly $q_1=1$. So further consider only the case $a_1+a_2+a_3<1$.
			Write $a_3=\frac{1}{3}-\delta$, and assume that $a_1+a_2+a_3<1$ and $0<\delta \leq \frac{1}{120}$ (which is equivalent to $a_3 \geq 0.325$).
			Note that
			
			\begin{equation}\label{3delta}
			1-a_1-a_2-a_3 \leq 1-3a_3=  3 \delta,
			\end{equation}
			and that 
			
			\begin{equation}\label{1delta}
			1-a_1-a_2+a_3,1-a_1+a_2-a_3 \leq 1-a_3= \frac{2}{3}+\delta.    
			\end{equation}


If $a_4 \leq 21 \delta \leq 7/40$, we have $\mathbf{a} \in \mathcal{A}_1$. So further assume that $a_4 \geq 21 \delta$, in which case we have to show that $\mathbf{a} \in \mathcal{A}_2 \cup \mathcal{A}_3$.

Let $k$ be the smallest integer such that $a_k<1-a_1-a_2-a_3$ (if $a_n \geq 1-a_1-a_2-a_3$, set $k=n+1$).
Note that $k \geq 5$, since
\[
a_4 \geq 21 \delta > 3\delta \geq 1-a_1-a_2-a_3,
\]
where the last inequality follows by \eqref{3delta}.

\begin{claim}\label{firststep}
If $\sum_{i=k}^{n}a_i^{2} \geq 450 \delta^{2}$, then $\mathbf{a} \in \mathcal{A}_2$.
\end{claim}

\begin{proof}[Proof of Claim \ref{firststep}]
Note that $a_k,\ldots,a_n<3 \delta$. We can find disjoint subsets $S,T_1,\ldots,T_4$ of $\lbrace a_k,\ldots,a_n \rbrace$ with the following properties. We have $$234 \delta^{2} \geq \sum_{i \in S}a_i^{2} \geq 225 \delta^{2}$$ and for $j=1,\ldots,4$, we have  $$54 \delta^{2} \geq \sum_{i \in T_j}a_i^{2} \geq 45 \delta^{2}.$$ Now consider the corresponding random processes $W(S;9 \delta)$ and $W(T_j;3 \delta)$ for $j=1,\ldots,4$.

We consider three events partitioning our probability space. The first event is the event $C_1$ that $W(S;9 \delta)$ is successful and also at least one out of $W(T_j;3 \delta)$ for $j=1,\ldots,4$ is successful. The second event is the event $C_2=C_2' \cap C_1^{C}$, where $C_2'$ is the event that at least one out of
\[
W(S;9 \delta),\ W(T_1;3 \delta),\ldots,\ W(T_4;3 \delta)
\]
is successful. And the last event is $C_3=C_1^{C} \cap C_2^{C}$.

By independence of the processes $W(S;9 \delta), W(T_1;3 \delta),\ldots,W(T_4;3 \delta)$ and Lemma~\ref{hittingprob}, we have 

\begin{equation}\label{variousc}
\pr \big[C_1\big] \geq \frac{15}{32},
\qquad\pr\big[C_3\big] \leq \frac{1}{32}.    
\end{equation}

We start by assessing the probability $\pr\big[|Y| \geq 1-a_1-a_2-a_3\big]$ conditioned on $C_1$.
We look at
\[
\pr \big[ |\sum_{i=4}^n a_i \varepsilon_i| < 1-a_1-a_2-a_3 | C_1,x_1,i_1,j,x_2,i_2,r(T_1;3 \delta),\ldots,r(T_{j-1};3\delta)\big],
\]
for fixed $x_1,i_1,j,x_2,i_2,r(T_1;3 \delta),\ldots,r(T_{j-1};3\delta)$, where $x_1,i_1,j,x_2,i_2$ are reals such that both
$
    |r(S;9\delta)|=x_1 \in [ 9\delta, 12 \delta ],
$
and the processes $W(T_1;3 \delta),\ldots,W(T_{j-1};3\delta)$ are not successful, but the process $W(T_j;3\delta)$ is successful for some fixed $j$, $1 \leq j \leq 4$, and \[
|r(T_j;3 \delta)|=x_2 \in [ 3\delta, 6 \delta ].
\]

Moreover, for the process $W(S;9\delta)$ it took $i_1$ terms to be successful, and for the process $W(T_j;3\delta)$ it took $i_2$ terms to be successful. Note that the value of $W(S;9\delta)$ is $\pm x_1$ with equal probabilities, and the value of $W(T_j;3\delta)$ is $\pm x_2$ with equal probabilities, independently both of each other and of all the other information.

Since $a_4 \geq 21 \delta$, we can apply Observation \ref{k3} with $a_4,x_1,x_2$ to conclude that 

\begin{equation}\label{eventc1}
\pr \big[|\sum_{i=4}^n a_i \varepsilon_i | < 1-a_1-a_2-a_3 | C_1,x_1,i_1,j,x_2,i_2,r(T_1;3 \delta),\ldots,r(T_{j-1};3\delta) \big] \leq \frac{1}{8}.
\end{equation}
As $x_1,i_1,j,x_2,i_2,r(T_1;3 \delta),\ldots,r(T_{j-1};3\delta)$ were arbitrary and we have finitely many possibilities for them, we conclude from \eqref{eventc1} that

\begin{equation}\label{c1}
\pr \big[|\sum_{i=4}^n a_i \varepsilon_i | < 1-a_1-a_2-a_3 | C_1 \big] \leq \frac{1}{8}.
\end{equation}

We can furthermore estimate the probability $\pr\big[|Y| \geq 1-a_1-a_2-a_3\big]$, conditioned on $C_2$, using Observation \ref{k2}:
\begin{equation}\label{c2}
\pr \big[|\sum_{i=4}^n a_i \varepsilon_i | < 1-a_1-a_2-a_3 | C_2 \big] \leq \frac{1}{4}.
\end{equation}
Analogously, the probability $\pr\big[|Y| \geq 1-a_1-a_2-a_3\big]$ conditioned on $C_3$, is significant, as shown by Observation~\ref{keyantichain}:
\begin{equation}\label{c3}
\pr \big[|\sum_{i=4}^n a_i \varepsilon_i| < 1-a_1-a_2-a_3 | C_3 \big] \leq \frac{1}{2}. \end{equation}
Combining \eqref{variousc}, \eqref{c1}, \eqref{c2} and \eqref{c3}, we get $q_1 \geq \frac{205}{256}> \frac{793}{1024}$, and hence $\mathbf{a} \in \mathcal{A}_2$.
\end{proof}

We turn to investigating the case $\sum_{i=k}^{n}a_i^{2} < 450 \delta^{2}$.
We record a property that will repeatedly be used in the sequel
\begin{equation}\label{first3}
\begin{split}
a_1^{2}+a_2^{2}+a_3^{2} &= (a_1+a_2+a_3-2a_3)^2+2a_3^2-2(a_1-a_3)(a_2-a_3) \\ &\leq (\frac{1}{3}+2\delta)^{2}+2(\frac{1}{3}-\delta)^{2}=\frac{1}{3}+6\delta^{2}.     
\end{split}
\end{equation}

\begin{claim}\label{kislarge}
If $\sum_{i=k}^{n}a_i^{2} < 450 \delta^{2}$, then $k \geq 11$.
\end{claim}

\begin{proof}[Proof of Claim \ref{kislarge}]
Assume that we had $\sum_{i=k}^{n}a_i^{2} < 450 \delta^{2}$ and $k \leq 10$. Then using \eqref{first3}, we get

\begin{equation*}
1=\sum_{i=1}^{n}a_i^{2} < (\frac{1}{3}+2\delta)^{2}+8(\frac{1}{3}-\delta)^{2} + 450 \delta^{2}=1+462 \delta^{2}-4\delta,
\end{equation*}
being a contradiction, as $1+462 \delta^{2}-4\delta < 1$ for $\delta \in (0, 1/120]$.
\end{proof}

\begin{claim}\label{higherboundaries}
If $k \geq 11$ and $a_8+a_9+a_{10} \geq \frac{2}{3}+ \delta$, then $\mathbf{a} \in \mathcal{A}_3$.
\end{claim}

\begin{proof}[Proof of Claim \ref{higherboundaries}]
Assume that we had $a_8+a_9+a_{10} \geq \frac{2}{3}+ \delta$ and $k \geq 11$.
Then by \eqref{1delta} and Observation \ref{keyantichain} applied to $a_4,\ldots,a_{10}$, we obtain $q_2,q_3 \geq \frac{37}{128}$.
\end{proof}

\begin{claim}\label{otherantichainused}
If $\sum_{i=k}^{n}a_i^{2} < 450 \delta^{2}$ and $a_5-a_{10} \geq 3 \delta$, then $\mathbf{a} \in \mathcal{A}_2$.
\end{claim}

\begin{proof}[Proof of Claim \ref{otherantichainused}]
Consider the events $D_1$, $D_2$, where $$D_1= \lbrace \varepsilon_4=\varepsilon_6=\varepsilon_7 \rbrace$$ and $D_2=D_1^{C}$. Note that 

\begin{equation}\label{dprobs}
\pr \big[D_1\big]=\frac{1}{4} \qquad \pr \big[D_2\big]=\frac{3}{4}.    
\end{equation}

In the case when $D_1$ occurs, let $c_1=a_4+a_6+a_7$, $c_2=a_5$, $c_3=a_{10}$. Since the conditions of Observation \ref{k3} hold for $c_1, c_2, c_3$ (by Claim~\ref{kislarge}, $a_{10} \geq 1-a_1-a_2-a_3$), we deduce that 

\begin{equation}\label{bettercase}
\pr \big[|\sum_{i=4}^{n} a_i \varepsilon_{i}|\leq 1-a_1-a_2-a_3\, |\, D_1 \big] \leq \frac{1}{8}.
\end{equation}

In the case when $D_2$ occurs, Observation \ref{k2} applied on $b_1=a_5$, $b_2=a_{10}$ implies that

\begin{equation}\label{worsecase}
\pr \big[|\sum_{i=4}^{n} a_i \varepsilon_{i}|\leq 1-a_1-a_2-a_3\,|\, D_2 \big] \leq \frac{1}{4}.    
\end{equation}
Combining \eqref{dprobs}, \eqref{bettercase} and \eqref{worsecase}, we get $$q_1 \geq \frac{25}{32} \geq \frac{793}{1024}.$$
\end{proof}

\begin{claim}\label{whatweconcludefrom}
If $\sum_{i=k}^{n}a_i^{2} < 450 \delta^{2}$, $a_5-a_{10} < 3 \delta$ and $a_8+a_9+a_{10} < \frac{2}{3}+ \delta$, then $k \geq 15$.
\end{claim}

\begin{proof}[Proof of Claim \ref{whatweconcludefrom}]
Assume that all of the conditions above hold, yet $k \leq 14$. We clearly have 

\begin{equation}\label{clearequation}
a_4 \leq a_3= \frac{1}{3}- \delta,    
\end{equation}
and the combination of $a_5-a_{10} < 3 \delta$ and $a_8+a_9+a_{10} < \frac{2}{3}+ \delta$ gives 

\begin{equation}\label{hardeqn}
a_{10},\ldots,a_{13} \leq \frac{2}{9}+\frac{\delta}{3} \qquad a_5,\ldots,a_9 < \frac{2}{9}+\frac{10}{3} \delta.     
\end{equation}
Using $\sum_{i=k}^{n}a_i^{2} < 450 \delta^{2}$, \eqref{first3}, \eqref{clearequation} and \eqref{hardeqn}, we get

\begin{align*}
1
&=\sum_{i=1}^{n}a_i^{2} \\
&=\sum_{i=1}^{3}a_i^{2}+a_4^{2}+\sum_{i=5}^{9}a_i^{2}+\sum_{i=10}^{k-1}a_i^{2} +\sum_{i=k}^{n}a_i^{2} \\
&\leq \Big(\frac{1}{3}+6 \delta^{2}\Big) + \Big(\frac{1}{3}-\delta\Big)^{2} + 5\Big(\frac{2}{9}+\frac{10}{3} \delta\Big)^{2} + 4\Big(\frac{2}{9}+\frac{\delta}{3} \Big)^{2}+450 \delta^{2} \\
&=\frac{8}{9}+513\delta^{2}+\frac{22}{3}\delta,
\end{align*}
being a contradiction, as the ultimate expression is strictly smaller than $ 1$ for any $\delta \in (0, 1/120]$. Hence $k \geq 15$.
\end{proof}

\begin{claim}\label{manysmall}
If $k \geq 15$, then $\mathbf{a} \in \mathcal{A}_2$.
\end{claim}

\begin{proof}[Proof of Claim \ref{manysmall}]
Applying Observation \ref{keyantichain} with $a_4,\ldots,a_{14}$ gives $q_1 \geq \frac{793}{1024}$, as required.      
\end{proof}
\noindent
The combination of the above claims concludes the proof of Lemma \ref{intothreeclasses}.
\end{proof}

We are now ready to complete the proof of~\eqref{eq:3/32} in the case $a_3 \geq 0.325$ and $a_1 + a_2 + a_3 \leq 1$; that is, we verify~\eqref{qcondition}.

We note that combining $\delta \leq \frac{1}{120}$ with \eqref{first3}, we get

\begin{equation}\label{scaling}
a_1^2+a_2^2+a_3^2 \leq \frac{801}{2400}.    
\end{equation}
First, consider the family $\mathcal{A}_1$ with $a_4 \leq \frac{7}{40}$. In this case,~\eqref{scaling} implies
\begin{equation}\label{scalingleading}
\frac{a_4}{\sqrt{\sum_{i=4}^{n} a_i^{2}}}  \leq \frac{\frac{7}{40}}{\sqrt{\frac{1599}{2400}}}  < 0.216.
\end{equation}
Moreover, using \eqref{3delta}:
\begin{equation}\label{scalingboundary}
\frac{1-a_1-a_2-a_3}{\sqrt{\sum_{i=4}^{n} a_i^{2}}}  \leq \frac{\frac{3}{120}}{\sqrt{\frac{1599}{2400}}}  < 0.032.    
\end{equation}
Finally,~\eqref{scalingleading} and~\eqref{scalingboundary} imply

\begin{equation}\label{typical}
q_1 \geq 2 D(0.216,0.032).    
\end{equation}

Analogously to \eqref{typical}, we have
\begin{equation*}
			\begin{gathered}
			q_2 \geq 2 D(0.216,0.828), \qquad
			q_3 \geq 2 D(0.216,0.828), \\
			q_4 \geq 2 D(0.216,0.858), \qquad
			q_5 \geq 2 D(0.216,1.634), \\
			q_6 \geq 2 D(0.216,1.654), \qquad
			q_7 \geq 2 D(0.216,1.654), \\
			q_8 \geq 2 D(0.216,2.452).
			\end{gathered}
\end{equation*}
Using the following estimate, 
\begin{gather*}
D(0.216,0.032)+D(0.216,0.828)+D(0.216,0.828)+D(0.216,0.858)+ D(0.216,1.634)+ \\ D(0.216,1.654)+D(0.216,1.654)+D(0.216,2.452)  \geq \frac{3}{4}   
\end{gather*}
we deduce~\eqref{qcondition} for any $\mathbf{a} \in \mathcal{A}_1$.

Next, we consider an $\mathbf{a}$ in the families $\mathcal{A}_2,\mathcal{A}_3$. Using $a_4 \leq \frac{1}{3}$ and \eqref{scaling}, we obtain

\begin{equation}\label{scalingleadingnew}
\frac{a_4}{\sqrt{\sum_{i=4}^{n} a_i^{2}}}  \leq \frac{\frac{1}{3}}{\sqrt{\frac{1599}{2400}}}  < 0.41,
\end{equation}
and we note that \eqref{scalingboundary} still holds.
Using \eqref{scalingboundary} and \eqref{scalingleadingnew}, we obtain 

\begin{equation}\label{typicalnew}
q_1 \geq 2 D(0.41,0.032).    
\end{equation}
Analogously to \eqref{typicalnew}, we derive

\begin{equation*}
			\begin{gathered}
			q_2 \geq 2 D(0.41,0.828), \qquad
			q_3 \geq 2 D(0.41,0.828), \\
			q_4 \geq 2 D(0.41,0.858), \qquad
			q_5 \geq 2 D(0.41,1.634), \\
			q_6 \geq 2 D(0.41,1.654), \qquad
			q_7 \geq 2 D(0.41,1.654), \\
			q_8 \geq 2 D(0.41,2.452).
			\end{gathered}
\end{equation*}

Note that we only mention the bound for $q_8$ above for the sake of completeness, since we have $D(0.41,2.452)=0$.

When $\mathbf{a} \in \mathcal{A}_2$ we can easily verify that 

\begin{gather*}
\frac{793}{2048}+D(0.41,0.828)+D(0.41,0.828)+D(0.41,0.858)+ D(0.41,1.634)+ \\ D(0.41,1.654)+D(0.41,1.654)+D(0.41,2.452)  \geq \frac{3}{4}   
\end{gather*}
and hence \eqref{qcondition} follows for all $\mathbf{a} \in \mathcal{A}_2$. For the family $\mathcal{A}_3$ we can verify that

\begin{gather*}
D(0.41,0.032)+\frac{37}{256}+\frac{37}{256}+D(0.41,0.858)+ D(0.41,1.634)+ \\D(0.41,1.654)+D(0.41,1.654)+D(0.41,2.452)  \geq \frac{3}{4}   
\end{gather*}
and hence \eqref{qcondition} follows for all $\mathbf{a} \in \mathcal{A}_3$. Proof of this subcase is thus finished. 



	\subsection{Case \tops{$a_1+a_2+a_3 > 1$}}
		\subsubsection{Subcase \tops{$a_1 + a_2 \geq 1$}}
			Using Observation \ref{keyantichain}, we have $\pr[|X| \geq 1] \geq 1/4$.

		\subsubsection{Subcase \tops{$a_1 \geq 0.7$} and not previous subcase}\label{ssec:0.7}
		    The proof is the same as in Section~\ref{ssec:0.72}.


        \subsubsection{Setting for the rest of the subcases}
        Assume $a_1 + a_2 < 1$ and $a_1+a_2+a_3 > 1$. The required inequality~\eqref{eq:3/32}, involves $\Pr[|X| \geq 1]$, and may be re-written using elimination in terms of $Y = \sum_{i=4}^{n} a_i \varepsilon_i$ as
		\begin{equation*}
		\begin{gathered}
			\frac{2}{8}\pr[|Y| \leq a_1+a_2+a_3-1] + \frac{1}{8}\pr[|Y| > a_1+a_2+a_3-1] +
			\frac{1}{8}\pr[|Y| \geq 1-a_1-a_2+a_3] + \\
			\frac{1}{8}\pr[|Y| \geq 1-a_1+a_2-a_3] + \frac{1}{8}\pr[|Y| \geq 1+a_1-a_2-a_3]	+ 
			\text{\small{$\frac{1}{8}\pr[|Y| \geq 1-a_1+a_2+a_3] +$}} \\
			\text{\small{$\frac{1}{8}\pr[|Y| \geq 1+a_1-a_2+a_3] + \frac{1}{8}\pr[|Y| \geq 1+a_1+a_2-a_3] + \frac{1}{8}\pr[|Y| \geq 1+a_1+a_2+a_3]$}}
			\geq 3/16.
		\end{gathered}
		\end{equation*}
		Denote
		\[
		L_1, L_2, L_3, L_4 = a_1+a_2+a_3-1, 1-a_1-a_2+a_3, 1-a_1+a_2-a_3, 1+a_1-a_2-a_3.
		\]
		
		The inequality we are proving follows by rearranging and multiplying the following inequality by $1/8$: 
		
		\begin{equation}\label{eq:geq1-need}
		\begin{gathered}
			\pr[|Y| \in (L_1, L_2)]
			\leq
			1/2 + \pr[|Y| \geq L_3] + \pr[|Y| \geq L_4].
		\end{gathered}
		\end{equation}
		Write $\sigma_j^2 = 1-\sum_{i=1}^{j} a_i^2$.
		Recall the variance of $Y$ is $\sigma_3^2$ and its largest weight is $a_4$.

		\subsubsection{Subcase \tops{$a_4 \geq 1-a_1-a_3$} and (either \tops{$a_4 \notin (L_1, L_2)$} or \tops{$\max(L_2-a_4, a_4-L_1) \leq 0.35 \sigma_4$}) and not previous subcases}
			Let us prove~\eqref{eq:geq1-need}, i.e. $\Pr[|Y| \in (L_1, L_2)] \leq 1/2 + \pr[|Y| \geq L_3] + \pr[|Y| \geq L_4]$.

			Since this inequality is symmetric with respect to $Y$, we may assume without loss of generality that $\varepsilon_4 = 1$, in which case it is clearly sufficient to prove
			\[
				\pr[Y \in (L_1, L_2) | \varepsilon_4 = 1] + \pr[Y \in (-L_2, -L_1) | \varepsilon_4 = 1] \leq 1/2 + \pr[Y > L_3 | \varepsilon_4 = 1].
			\]
			To this end, note that $\pr[Y \in (-L_2, -L_1) | \varepsilon_4 = 1] \leq \pr[Y > L_3 | \varepsilon_4 = 1]$, which follows by (recall $L_3 - a_4 \leq a_4 + L_1$ by assumption):
			\[
			    \Pr[Y'+a_4 < -L_1] = \Pr[Y' > L_1 + a_4] \leq \Pr[Y' > L_3 - a_4]
			\]
			with $Y' = Y - a_4 \varepsilon_4$.
		
			Hence our task is to verify $\pr[Y \in (L_1, L_2) | \varepsilon_4 = 1] \leq 1/2$. There are two subcases. If $a_4 \leq L_1$ or $a_4 \geq L_2$, then we conclude with a general $\Pr[Y' > 0] \leq 1/2$ bound. If $a_4 \in [L_1, L_2]$, we conclude with the inequality from Section~\ref{ssec:sqrt7}, recalling that $\max(L_2-a_4, a_4-L_1) \leq 0.35 \sigma_4$.


		\subsubsection{Subcase not previous cases}
		    Note that~\eqref{eq:geq1-need} follows from
		    \[
		        \Pr[|Y| > L_1] \leq 1/2 + \Pr[|Y| \geq L_2] + \Pr[|Y| \geq L_3] + \Pr[|Y| \geq L_4].
		    \]
		    As the left hand side is a probability, it is sufficient we show the right hand side is at least $1$. This in turn follows from (see Appendix~\ref{app:0.25})
		    \begin{equation}\label{eq:0.25}
		        D(a_4/\sigma_3, L_2/\sigma_3) + D(a_4/\sigma_3, L_3/\sigma_3) + D(a_4/\sigma_3, L_4/\sigma_3) \geq 1/4.
		    \end{equation}


\section{Proof of \tops{$\Pr[X > 1] \geq 1/16$} unless \tops{$X = \varepsilon_1$}}\label{sec5}

In this section, we prove Theorem \ref{thm:intro-olesz2} (which is the best possible). Note that unlike for Theorem \ref{thm:intro-olesz} where significant further work was required, most of the work toward proving Theorem \ref{thm:intro-olesz2} was done when we developed our tools in \ref{dynpro} and now we can just conclude pretty easily.

\subsection{Case \tops{$a_1 + a_2 + a_3 > 1$}}
Clearly,
\[
    \Pr[X > 1] \geq \Pr\li[\sum_{i=1}^{3} a_i \varepsilon_i > 1 \andd \sum_{j=4}^{n} a_j \varepsilon_j\geq 0\ri] \geq 1/8 \cdot 1/2 = 1/16.
\]

\subsection{Case \tops{$a_1 + a_2 + a_3 \leq 1$}}
    In this case we actually show $\Pr[X > 1] \geq 1/12$, and the proof is analogous to that of Section~\ref{ssec:0.31}.
    
    
    \subsubsection{Subcase \tops{$a_1 \leq 0.4$}}
        We conclude using \eqref{eq:DD} and \eqref{eq:stash-of-D} with $D(0.4, 1) > \frac{1}{12}$.
    
    \subsubsection{Subcase \tops{$a_1 \geq 0.6$}}
        Let $a = 0.6$. We conclude using elimination, \eqref{eq:DD} and \eqref{eq:stash-of-D} with
        \[
            D\left(\frac{1-a}{\sqrt{1-a^2}}, \frac{1-a}{\sqrt{1-a^2}} \right) = D(1/2, 1/2) > 1/6.
        \]
        Notice that in this case $a_1$ might be $1$, which forbids elimination by $a_1$. This is where the assumption $X \neq \varepsilon_1$ is used.
    
    \subsubsection{Subcase \tops{$a_1 \in [0.4, 0.6]$}}
    We write $\sigma_2 = \sqrt{1-a_1^2-a_2^2}$ and recall that $a_3$ is upper bounded by $a=\min(a_2, 1-a_1-a_2)$, so that $\Pr[X > 1] \geq 1/12$ follows from (see Appendix~\ref{app:16}):
    \begin{equation}\label{eq:16}
        \be_{\varepsilon \in \spm^2} \li[ D\li( \frac{a}{\sigma_2}, \frac{1 + a_1 \varepsilon_1 + a_2 \varepsilon_2}{\sigma_2}\ri) \ri] \geq 1/12.
    \end{equation}


\section{Toward the \tops{$7/64$} bound}\label{sec6}

We strongly believe that $C_1=7/64$. Further to the brief discussion in subsection \ref{difcases}, we will comment in this section what the next steps would be and what hurdles one would face if we try to continue further to this bound using the methods of this paper, i.e. combining lower bounds of the type \ref{dynpro} with separate arguments for some difficult cases. While somewhat tedious, we note that similar approach was recently used by Keller and the second author to resolve the problem of Tomaszewski \cite{KK20}. Nevertheless, the tools needed here would be rather different than the ones used in the proof of Tomaszewski's conjecture, since we are now dealing with an anti concentration inequality instead of a concentration one.


Continuing further to the $7/64$ bound using our methods (or similar ones), there are two particular classes of the collections $\lbrace a_i \rbrace$ one has to be very careful about.

First such class are the collections $\lbrace a_i \rbrace$ for which we have precisely $\pr \big[ X \geq 1 \big]=\frac{7}{64}$ and thus we can not afford to obtain any suboptimal bound. As an example of the collection in the first class, one can consider $a_1=\ldots=a_6=\frac{1}{\sqrt{6}}$. For this particular collection, the bound follows trivially from Observation \ref{keyantichain}, since $a_3+a_4+a_5 \geq 1$. We suspect that in fact all the collections in this class satisfy $a_3+a_4+a_5 \geq 1$, making it not too difficult to handle.

Second such class are the collections $\lbrace a_i \rbrace$ with $$\pr \big[ X>1 \big] < \frac{7}{64},$$ since for these one can't verify the conjecture by only assuming that the few largest weights lie in some, however narrow, ranges. Five examples of the collections in the second class are mentioned in the subsection \ref{difcases} and we believe these are only such examples.



The collections ‘close to' $a_1=1$ are not a big problem for us, since Lemma \ref{hittingprob} allows us to show that the bound of $7/64$ holds for collections with $a_1$ large.

\begin{proposition}
If $a_1 \geq \frac{14}{15}$, then $\pr \big[ X \geq 1 \big] \geq \frac{7}{64}$.    
\end{proposition}

\begin{proof}
Note that it is enough to argue that $p(a_2,\ldots,a_n;1-a_1) \geq \frac{7}{8}$. For that, by Lemma \ref{hittingprob} we know that it suffices if $1-a_1^2 \geq 29(1-a_1)^2$. We can easily check that this is satisfied whenever $a_1 \geq \frac{14}{15}$.   
\end{proof}

‘Neighbourhoods' of the remaining problematic collections are more difficult (though luckily note that the family $F_1(\delta)$ below covers the ‘neighbourhood' of both the second and the fifth collection). For fixed $\delta>0$, consider the families $$F_1(\delta)= \lbrace a_1+a_2<1; a_2 \geq \frac{1}{2} - \delta \rbrace,$$ $$F_2(\delta)= \lbrace a_1+a_2<1; |a_1-\frac{2}{3}|, |a_2-\frac{1}{3}| \leq \delta \rbrace,$$ $$F_3(\delta)= \lbrace a_1+a_2+a_3<1; a_3 \geq \frac{1}{3} - \delta \rbrace.$$ If we want to verify that $C_1=7/64$ with the help of computational methods similar to the ones used in this paper, we must be able to find some $\delta>0$ for which we can verify by different means that the conjecture holds for all the collections in $F_1(\delta), F_2(\delta),F_3(\delta)$. Hope is this could be done in somewhat similar way as the proof of $6/64$ bound within $F_3(\frac{1}{120})$ in \ref{hardsubsection} when proving Theorem \ref{thm:intro-olesz}.

We make a progress in that direction by using stopped random walks and chain arguments to prove the following.

\begin{proposition}\label{combinedprops}
For $\delta_0=10^{-9}$, we have $\pr \big[ X \geq 1 \big] \geq \frac{7}{64}$ for all collections $\lbrace a_i \rbrace$ in $ F_1(\delta_0), F_2(\delta_0)$.    
\end{proposition}

Our value $\delta_0$ is extremely small, but that is because we have not tried to optimize it at all (as that would result in an even more tedious argument). We believe with some effort, our solution could be improved to work for much larger value of $\delta$ which could actually be used in practice. 

The arguments for $F_1(\delta_0)$ and $F_2(\delta_0)$ are rather similar in style and are somewhat tedious. Hence in this section, we only include the argument for the family $F_1(\delta_0)$ and the argument for the family $F_2(\delta_0)$ is placed in Appendix \ref{fam2}.

Surprisingly, while we were able to improve the bound closer to $\frac{7}{64}$ in that case too, we were not able to prove the bound of $7/64$ for the family $F_3(\delta)$ for any $\delta>0$, so we pose this as an open problem to the reader. We believe even verifying the conjecture just in this narrow range of parameters would be of interest.








In subsection \ref{fami1} and in Appendix \ref{fam2}, we sometimes sketch the proofs instead of going through all the details of the calculations. That is because the calculations would otherwise be very long and it is easy to see that the sketch could indeed be turned into a rigorous proof.

\subsection{Family \tops{$F_1(\delta_0)$}}\label{fami1}

In this subsection, we prove the following result.

\begin{proposition}\label{firstpart}
For $\delta_0=10^{-9}$, we have $\pr \big[ X \geq 1 \big] \geq \frac{7}{64}$ for all collections $\lbrace a_i \rbrace$ in $ F_1(\delta_0)$.    
\end{proposition}

Together with Proposition \ref{secondpart}, this implies Proposition \ref{combinedprops}.

Assume $a_1+a_2<1$ and $a_2=\frac{1}{2}-\delta$ for some $\delta \leq 10^{-9}$. Also assume our collection $ \lbrace a_1,\ldots,a_n \rbrace$ has $\pr \big[ |X| \geq 1 \big]<\frac{7}{32}$. We will derive a contradiction.

Note that $1-a_1-a_2 \leq 2 \delta$.
Denote $Y = \sum_{i=3}^{n} a_i \varepsilon_{i}$ and
\begin{equation*}
\begin{gathered}
p_1=\pr\big[|Y|\geq 1-a_1-a_2\big],\qquad
p_2=\pr\big[|Y|\geq 1-a_1+a_2\big],\qquad
p_3=\pr\big[|Y|\geq 1+a_1-a_2\big].
\end{gathered}
\end{equation*}
Then, in particular, we have
\[
    \pr \big[ |X| \geq 1 \big] \geq \frac{1}{4}(p_1+p_2+p_3).
\]
So, it is enough to show
\begin{equation}\label{goalf1}
p_1+p_2+p_3 \geq \frac{7}{8}.    
\end{equation}
We can also assume that
\begin{equation}\label{verysimple}
a_3+a_4+a_5<1,     
\end{equation}
else we would be done by Observation \ref{keyantichain}. We will make consecutive claims about $\lbrace a_1,\ldots,a_n \rbrace$, characterizing it more and more precisely until we are ready to obtain a contradiction. 

Call $a_i$ big if $a_i \geq 1-a_1-a_2$, and small otherwise. So in particular if $a_i \geq 2 \delta$, it must be big. Let $k$ be the smallest integer such that $a_k<1-a_1-a_2$ (if $a_n \geq 1-a_1-a_2$, set $k=n+1$). 

\begin{claim}\label{aroundhalffirst}
Let $k$ be the smallest integer such that $a_k<1-a_1-a_2$. Then we have $\sum_{i=k}^{n}a_i^{2} \leq 240 000 \delta^{2} < \frac{\delta}{1000}$.
\end{claim}

\begin{proof}
Assume for contradiction that this is not true. Then we can take disjoint subsets $S_1,\ldots,S_{10000}$ of $\lbrace a_k,\ldots,a_n \rbrace$ with $$24 \delta^{2} \geq \sum_{i \in S_j}a_i^{2} \geq 20 \delta^{2}$$ for $j=1,\ldots10000$. Now considering the random processes $W(S_i;2 \delta)$ for $i=1,\ldots,10000$, with probability at least $\frac{99}{100}$, at least $1000$ of these are successful, and conditional on that, we obtain $\pr\big[|Y|\geq 1-a_1-a_2\big] \geq \frac{19}{20}$ by Observation \ref{keyantichain}. Hence we overall get $$ p_1 \geq \frac{1881}{2000}>\frac{7}{8},$$ and \eqref{goalf1} holds.
\end{proof}

\begin{claim}\label{exampleone}
$a_5$ and $a_6$ are big terms, that is, $a_6 \geq 1-a_1-a_2$.
\end{claim}

\begin{proof}
Assume for contradiction that $a_6$ is a small term (i.e. that $k \leq 6$).
Combining Claim \ref{aroundhalffirst} with \eqref{verysimple}, we arrive at a contradiction for all sufficiently small $\delta > 0$:
\[
1 = \sum_{i=1}^{5} a_i^2 + \sum_{i=6}^{n} a_i^2 \leq a_1^2 + a_2 \sum_{i=2}^{5} a_i + \frac{\delta}{1000} \leq (1/2+\delta)^2 + (1/2-\delta)(3/2-\delta) + \frac{\delta}{1000} = 1-\frac{999}{1000}\delta + 2\delta^2.
\]

%
%
%
\end{proof}

At this point, we split our proof into two cases, the uniform and the non-uniform one, both of which we handle separately.

\subsubsection{The uniform case - \tops{$a_3-a_{k-1} \leq 20 \delta$} }

\begin{claim}
Let $k$ be the smallest integer such that $a_k<1-a_1-a_2$. Then we have $k \leq 11$.
\end{claim}

\begin{proof}
Assume we had $k \geq 12$. Note that $a_3-a_{k-1} \leq 20 \delta$ would then in particular imply

\begin{equation}\label{ingredient1}
a_3 \leq \sqrt{0.5/9}+O(\delta) < 0.24,    
\end{equation}
and we also know 

\begin{equation}\label{ingredient2}
\sum_{i=3}^{n}a_i^{2} \geq 0.4999    
\end{equation}
and 

\begin{equation}\label{ingredient3}
1+a_1-a_2 \leq 1.00001.    
\end{equation}

Using Observation \ref{keyantichain} for $a_3,\ldots,a_{11}$, we get $$p_1\geq \frac{386}{512} > \frac{3}{4}.$$ Combining \eqref{ingredient1}, \eqref{ingredient2} and \eqref{ingredient3}, and using \eqref{eq:stash-of-D}, we get $$p_2,p_3 \geq 2D(0.34, 1.42) > 0.08 > \frac{1}{16}$$ and hence \eqref{goalf1} holds. 
\end{proof}

The next corollary follows by combining Chebyshev inequality with Claim \ref{aroundhalffirst}, using that $\delta$ is small.

\begin{corollary}\label{resttiny}
Let $k$ be the smallest integer such that $a_k<1-a_1-a_2$. Then we have $\pr \big[ |\sum_{i=k}^{n} a_i \varepsilon_i | \geq 0.0001 \big]<\frac{1}{1000}$.
\end{corollary}

We now sketch how we finish our argument in the subcase $a_3-a_{k-1} \leq 20\delta$, using Corollary \ref{resttiny}. We consider five separate cases depending on the particular value of $k$ which we know is at least $7$ and at most $11$ (and in fact, we can rule out the case $k=7$ as then we would have $a_3+a_4+a_5 \geq 1$). Due to our restrictions on the value of $\delta$ and Corollary \ref{resttiny}, we know that $\sum_{i=3}^{n} a_i \varepsilon_i $ behaves ‘essentially' like $\sum_{i=3}^{k-1} \varepsilon_i \frac{1}{\sqrt{2k-6}}$. So for instance in the case $k=8$, we argue that $p_1 \geq \frac{999}{1000}$, as due to our restrictions on $a_3,\ldots,a_7$, we know we can only have $|\sum_{i=3}^{n} a_i \varepsilon_i  | < 2 \delta$ if $|\sum_{i=8}^{n} a_i \varepsilon_i | \geq 0.0001$; further, in this case $k=8$, we analogously argue that $p_2,p_3 \geq \frac{999}{1000} \cdot \frac{1}{32}$. 

Similarly, in the case $k=9$, we argue that $p_1 \geq  \frac{11}{16}$, $p_2,p_3 \geq \frac{999}{1000} \cdot \frac{7}{32}$.

The reader can easily verify that such arguments indeed work in all the cases considered. $\square$

\subsubsection{The non-uniform case - \tops{$a_3-a_{k-1}>20 \delta$}}

In this case, we first notice that Observation \ref{k2} applied to $a_3,a_{k-1}$ immediately implies the following.

\begin{claim}\label{threequartercor}
We have $p_1 \geq \frac{3}{4}$.
\end{claim}





Next we obtain.

\begin{claim}
We have $a_3+a_4+a_5+a_6 < 1+2 \delta$.
\end{claim}

\begin{proof}
Assume not. Then by Observation \ref{keyantichain}, we have $p_2,p_3 \geq \frac{1}{16}$, and combining this with Claim \ref{threequartercor} gives \eqref{goalf1}.        
\end{proof}

\begin{claim}\label{exampletwo}
$a_7$ is a big term.
\end{claim}

\begin{proof}
Assume for contradiction that $a_7$ is a small term (i.e. that $k \leq 7$), and recall that $a_3+a_4+a_5+a_6<1+2\delta$ and $\sum_{i=7}^{n}a_i^{2} < 240 000 \delta^{2} < \frac{\delta}{1000}$. Write $A = a_3+a_4$, and arrive at a contradiction for all sufficiently small $\delta > 0$:
\begin{equation*}
\begin{aligned}
1 &= \sum_{i=1}^{6} a_i^2 + \sum_{i=7}^{n} a_i^2
\leq a_1^2 + a_2^2 + a_2 (a_3+a_4) + a_5 (a_5+a_6) + \frac{\delta}{1000} \\
& \leq  \Big(\frac{1}{2}+\delta\Big)^2 + \Big(\frac{1}{2}-\delta\Big)^2 + \Big(\frac{1}{2}-\delta\Big)A + \frac{1+2\delta}{3}(1+2\delta-A) + \frac{\delta}{1000} \\
& = \frac{1}{2}+2\delta^2 + A\Big (\frac{1}{6}-\frac{5}{3}\delta \Big) + \frac{1}{3}(1+2\delta)^2 + \frac{\delta}{1000} \\
& \leq \frac{1}{2}+2\delta^2 + (1-2\delta)\Big (\frac{1}{6}-\frac{5}{3}\delta \Big) + \frac{1}{3}(1+2\delta)^2 + \frac{\delta}{1000} \\
& = 1 - \frac{2}{3}\delta + \frac{20}{3}\delta^2 + \frac{\delta}{1000} < 1.
\end{aligned}
\end{equation*}
where we used the estimates $a_5+a_6\leq 1+2\delta-A$, and $a_5 \leq (1+2\delta)/3$ and $A \leq 2a_2 \leq 1-2\delta$.

%
%
%
%
\end{proof}

\begin{claim}\label{nothingtoosmall}
We have $a_4 \geq 0.07$.
\end{claim}

\begin{proof}
If not, we can use Claim \ref{aroundhalffirst} to argue that we have at least $44$ big terms, otherwise we would have $$\sum_{i=3}^{k-1} a_i^2 < 0.49.$$ But Observation \ref{keyantichain} then implies $p_1 \geq \frac{7}{8}$, and \eqref{goalf1} follows.        
\end{proof}

\begin{claim}\label{p2p3}
We have $p_2,p_3 \geq \frac{3}{64}$.
\end{claim}

\begin{proof}
We consider two cases. If $a_3 \leq 0.3$, the result follows using the bounds \eqref{ingredient2} and \eqref{ingredient3} as well as \eqref{eq:stash-of-D} by $$p_2,p_3 \geq 2D(0.43, 1.42) > 0.06 > \frac{3}{64}.$$

If on the other hand $a_3>0.3$, we may argue (using Claim \ref{nothingtoosmall} and argument much along the same lines as the proofs of Claim \ref{exampleone} and Claim \ref{exampletwo}) that $$a_3+a_4+\sqrt{\sum_{i=5}^{n}a_i^{2}} \geq 1+a_1-a_2.$$ But then let $\varepsilon'$ be a sign of $\sum_{i=5}^{n} a_i \varepsilon_i $, and consider the events

\begin{equation*}
A= \lbrace \varepsilon \cc \varepsilon_3=\varepsilon_4=\varepsilon' \rbrace, \qquad B=\bigg\lbrace \varepsilon \cc |\sum_{i=5}^{n} a_i \varepsilon_i | \geq \Big( \sum_{i=5}^{n}a_i^{2} \Big)^{1/2} \bigg\rbrace.
\end{equation*}

We have $\pr\big[A \cap B\big] \geq \frac{3}{64} $ (using our bound from the previous sections), and clearly $$|\sum_{i=3}^{n} a_i \varepsilon_i| \geq 1+a_1-a_2$$ whenever event $A \cap B$ occurs. The result follows.      
\end{proof}

\begin{claim}\label{lastingredient}
Let $k$ be the smallest integer such that $a_k<1-a_1-a_2$. Then we have $a_4-a_{k-1} \leq 2 \delta$.
\end{claim}

\begin{proof}
Assume for contradiction that $a_4-a_{k-1} \geq 2 \delta$. Then $a_3+a_5+a_6$
is not within $2 \delta$ neither from $a_4$ nor from $a_4+a_{k-1}$.
Using Observation \ref{k2} for $a_4,a_{k-1}$ in the case when we do not have $\varepsilon_3=\varepsilon_5=\varepsilon_6$, and Observation \ref{k3} for $a_4,a_{k-1},a_3+a_5+a_6$ in the case when we have $\varepsilon_3=\varepsilon_5=\varepsilon_6$ (which happens with probability $\frac{1}{4}$) gives

\begin{equation}\label{almostend}
p_1 \geq \frac{3}{4} \cdot \frac{3}{4}+\frac{1}{4} \cdot \frac{7}{8}=\frac{25}{32}.    
\end{equation}
Combining Claim \ref{p2p3} with \eqref{almostend} gives \eqref{goalf1}. 
\end{proof}

Now we are ready to reach the contradiction. First, if $a_3 \notin (2a_{k-1}-8\delta,2a_{k-1}+8\delta)$, let $f_1=a_3+a_{k-1}$ and $f_2=a_4+a_5$. Let $A_1= \lbrace \varepsilon_3=\varepsilon_{k-1} \rbrace$ and $A_2= \lbrace \varepsilon_4=\varepsilon_5 \rbrace$. Then conditional on $A_1 \cap A_2$, we have $\pr\big[|Y|\geq 1-a_1-a_2\big] \geq \frac{7}{8}$ by Observation \ref{k3} for $f_1,f_2,a_6$; conditional on $A_1 \cap A_2^{C}$, we have $\pr\big[|Y|\geq 1-a_1-a_2\big] \geq \frac{3}{4}$ by Observation \ref{k2} for $a_3+a_{k-1},a_6$; and conditional on $A_1^{C}$, we have $\pr\big[|Y|\geq 1-a_1-a_2\big] \geq \frac{3}{4}$ by Observation \ref{k2} for $a_3-a_{k-1},a_6$. So we conclude $p_1 \geq \frac{25}{32}$, and hence \eqref{goalf1} holds.

So next assume $a_3 \in (2a_{k-1}-8\delta,2a_{k-1}+8\delta)$. Here, we observe that we can assume $k \leq 15$, else we could conclude $p_1 \geq \frac{25}{32}$ from Observation \ref{keyantichain}. But now, we proceed analogously to how we did at the end of the argument for the uniform case, again using Corollary \ref{resttiny} and detailed analysis of each of the several cases we have depending on the value of $k$. Carrying out such analysis is made possible by Claim \ref{lastingredient}.

So the proof of Proposition \ref{firstpart} is complete. $\square$

\section{The high-dimensional version of the problem}\label{sec7}


	The following (non-tight) result constitutes a high-dimensional variant of Tomaszewski's problem as well as of the problem studied in this paper. The result is merely a consequence of the combination of ~\cite[Proposition~2.2]{MV08} and~\cite[Theorem~2]{IT18}. Nevertheless, for the sake of completeness, we prove it here.
	\begin{proposition}\label{prop:0.035}
	    Let $v_1, \ldots, v_n \in \reals^d$ be vectors with $\sum_i \norm{v_i}_2^2 = 1$. The random variable $X = \sum v_i \varepsilon_i$ with $\varepsilon_i \sim \spm$ uniformly and independently distributed, satisfies
	    \[
	        \pr[\norm{X}_2 \geq 1] \geq \frac{1-\sqrt{1-1/e^2}}{2} > 0.035, \qquad\qquad \pr[\norm{X}_2 \leq 1] \geq \frac{1-\sqrt{1-1/e^2}}{2}.
	    \]
	\end{proposition}
	\begin{proof}
        The function $f(\varepsilon) = \norm{X(\varepsilon)}_2^2 - 1 = \sum_{i,j} \varepsilon_i \varepsilon_j \li\langle v_i, v_j\ri\rangle$ is a \emph{homogenuous} polynomial of degree $2$ in the $\varepsilon_i$'s. We wish to lower bound the probabilities $\pr[f(\varepsilon) \geq 0]$ and $\pr[f(\varepsilon) \leq 0]$.
        Recall~\cite[Theorem~2]{IT18}:
        \begin{equation}\label{eq:75}
            \norm{f}_2 \leq e \norm{f}_1.
        \end{equation}
        Since $\be[f] = 0$, we can derive (see below)
        \begin{equation}\label{eq:76}
            \norm{f}_1^2 \leq 4 \pr[f > 0] \pr [f \leq 0] \norm{f}_2^2.
        \end{equation}
        Plugging~\eqref{eq:75} into~\eqref{eq:76} we get
        \[
        \norm{f}_1^2 \leq 4 e^2 \pr[f > 0] \pr [f \leq 0] \norm{f}_1^2.
        \]
        When $f\equiv 0$, we have $\Pr[f = 0] = 1$. Otherwise, dividing by $\norm{f}_1^2$ we obtain $$\Pr[f > 0]\pr[f \leq 0] \geq e^{-2}/4,$$ which means both $\pr[f > 0]$ and $\pr[f \leq 0]$ are at least $\frac{1-\sqrt{1-1/e^2}}{2}$, through $\pr[f>0]+\pr[f \leq 0] = 1$.
        
        To see~\eqref{eq:76}, notice that by the Cauchy-Schwarz inequality,
        \begin{equation}\label{eq:78}
            \norm{f}_2^2 = \be[f^2 \cdot \one\{f > 0\}] + \be[f^2 \cdot \one\{f \leq 0\}] \geq
            \pr[f > 0] \be\li[\given{f}{f > 0}\ri]^2 + \pr[f \leq 0] \be\li[\given{f}{f \leq 0}\ri]^2.
        \end{equation}
        As $\be[f]=0$, we have $$\be[f \cdot \one\{f > 0\}] = -\be[f \cdot \one\{f \leq 0\}]=\frac{1}{2}\norm{f}_1.$$ Likewise, we may assume $\pr[f > 0]$ and $\pr[f \leq 0]$ are both positive as otherwise~\eqref{eq:76} trivially holds. Under this assumption,~\eqref{eq:78} yields
        \[
        \norm{f}_2^2 \geq \frac{1}{4} \norm{f}_1^2 \li( \frac{1}{\pr[f > 0]} + \frac{1}{\pr[f \leq 0]} \ri),
        \]
        being~\eqref{eq:76}, using again $\pr[f > 0] + \pr[f \leq 0] = 1$.
	\end{proof}


Denote by $T_d$ the maximum constant for which $\pr[\norm{X}_2 \leq 1] \geq T_d$ for all $X$ of dimension $d$ as in Proposition~\ref{prop:0.035}, and denote by $O_d$ the maximum constant for which $\pr[\norm{X}_2 \geq 1] \geq O_d$ for all $X$ of dimension $d$ as in Proposition~\ref{prop:0.035}. Clearly, $T_d$ and $O_d$ are non-increasing in $d$. We know $T_1=\frac{1}{2}$ \cite{KK20}, while this paper proves that $\frac{6}{32} \leq O_1 \leq \frac{7}{32}$. Proposition~\ref{prop:0.035} establishes that $T_d,O_d \geq 0.035$ for any $d$. There are two directions for further research here. 

The first is to find tighter bounds for $T_d,O_d$ for small values of $d > 1$. We know that $T_2 \leq \frac{1}{4}$, as demonstrated by
\[_1=(\frac{1}{\sqrt{3}},0), \qquad v_2=(-\frac{1}{2 \sqrt{3}}, \frac{1}{2}), \qquad v_3=(-\frac{1}{2 \sqrt{3}},-\frac{1}{2}),
\]
and $T_3 \leq \frac{3}{16}$, as demonstrated by \[
v_1=(\sqrt{\frac{7}{30}},\frac{1}{3},\frac{1}{5}), \ \quad v_2=(\sqrt{\frac{7}{30}},-\frac{1}{3},-\frac{1}{5}),\ \quad v_3=(0,\frac{1}{3},-\frac{1}{5}), \ \quad v_4=(0,0,\frac{1}{5}), \ \quad v_5=(0,0,\frac{1}{5}).
\]

Interestingly, we have not been able to find any examples demonstrating that $O_2<\frac{7}{32}$ (or even that $O_{d_0}<\frac{7}{32}$ for any $d_0 \in \mathbb{N}$), and hence we pose this as a problem to a reader.

The second possible direction is to investigate how $T_d,O_d$ behave for large $d$, and in particular to find better bounds for $\inf_d T_d$ and $\inf_d O_d$. It appears that Proposition~\ref{prop:0.035} is far from being tight. Also, as just mentioned, it does not seem completely unthinkable that $O_d = \frac{7}{32}$ for every $d \in \mathbb{N}$ could hold.

\section{Conclusion}\label{sec8}

As mentioned previously, we would hope that mixed with some new ideas, the methods developed in this paper could be used to prove the conjectured optimal bound of $7/64$ in Theorem \ref{thm:intro-olesz}. That is the main open problem left, and even some progress toward that (like improving Theorem \ref{thm:intro-olesz} to hold for some constant between $6/64$ and $7/64$) would be of interest. 

Another, perhaps easier step one could take in this direction would be to prove the bound of $7/64$ for the ‘difficult' family $F_3(\delta_0)$ for some $\delta_0>0$. The significance of this is discussed in more detail in Section \ref{sec6}.  

In a bit different direction, it is likely that one could improve the multiplicative factor in front of $\sqrt{\var(X)}$ in Theorem \ref{thm:intro-lowther} from $0.35$ to the optimal conjectured \cite{lowther20} value of $1/\sqrt{7}$. That would not only be of interest on its own, but as demonstrated by this paper and our use of Theorem \ref{thm:intro-lowther} when deriving Theorem \ref{thm:intro-olesz}, also a useful tool when attacking similar problems.

Finally, let us mention two interesting generalizations of our main problem that one can consider. 

Firstly, same as Keller and the second author \cite{KK20}, we ask what is the behaviour of the function
$$ F(x)  = \sup_{X} \Pr[X > x],$$
where the supremum is taken over all the Rademacher sums with variance $1$. Theorem \ref{thm:intro-olesz} establishes that $F(-1) \leq \frac{58}{64}$. We know some asymptotic results about the behaviour of $F(x)$ \cite{pinelis2012asymptotically} and we also know the precise value of $F(x)$ for some $x$ \cite{bentkus2015tight,KK20,pinelis2012asymptotically}, but much remains to be understood. It would be tempting to conjecture that $F(x)=F^{=}(x)$, where for $F^{=}(x)$, we take the supremum over all the the Rademacher sums with variance $1$ and all the weights equal. Nevertheless, this conjecture turns out not to be true, see \cite{pinelis2015supremum}.

Further, one can also study the various multi-dimensional questions that arise, as discussed in Section \ref{sec7}. We find it especially intriguing that we have not managed to find any $d_0 \in \mathbb{N}$ for which we could show that $O_{d_0}< 7/32$. If there is no such $d_0$, that would be a beautiful generalization of the result of the one dimensional version of the problem.

\section*{Acknowledgements}
We profoundly thank B\'ela Bollob\'as, Nathan Keller, Peter van Hintum, Marius Tiba and the anonymous referees for fruitful discussions and suggestions.

The first author was supported by EPSRC (grant no. 2260624). The second author was supported by the Clore Scholarship Programme, and by the Israel Science Foundation (grant no. 1612/17).

\bibliographystyle{abbrv}
\bibliography{refe}	
	
	\begin{appendices}
	\section{Proofs of real numbers inequalities}\label{appA}
	
	    
	\subsection{Proof of~\tops{\eqref{eq:0.325}}}\label{app:0.325}
	    We consider only these $a_1, a_2$ with $a_1+a_2 \leq 1$ and $a_2 \leq a_1 \in [0.3, 0.7]$. We denote $a = \min(1-a_1-a_2, a_2, 0.325)$ (being an upper bound on $a_3$), and $\sigma_2 = \sqrt{1-a_1^2-a_2^2}$. We note that both $a/\sigma_2$ and $(1+a_1\varepsilon_1 + a_2\varepsilon_2)/\sigma_2$ for any choice of $\varepsilon_1, \varepsilon_2 \in \spm^n$ are $10$-Lipschitz in our domain (e.g., by checking that all partial derivatives $< \sqrt{50}$ in absolute value), so it suffices we check
	    \begin{equation}\label{eq:0.31ver}
	    \be_{\varepsilon \in \spm^2} \li[ D\li( \frac{a}{\sigma_2}+\delta, \frac{1 + a_1 \varepsilon_1 + a_2 \varepsilon_2}{\sigma_2} +\delta\ri) \ri] \geq 3/32
	    \end{equation}
	    on a mesh of $\set{(a_1, a_2)}{a_1+a_2 \leq 1, a_2 \leq a_1 \in [0.3, 0.7]}$ of granularity $\delta / 10$ in both axes. Inequality~\eqref{eq:0.31ver} can easily be verified \cite{us} for $\delta=0.005$ on such a mesh. 
	    

	\subsection{Proof of~\tops{\eqref{eq:0.25}}}\label{app:0.25}
	    In order to verify~\eqref{eq:0.25} for all relevant $a_1, a_2, a_4, a_4$, we confirm
	    \begin{equation}\label{eq:0.25pp}
	        D(a_4/\sigma_3+\delta, L_2/\sigma_3+\delta) + D(a_4/\sigma_3+\delta, L_3/\sigma_3+\delta) + D(a_4/\sigma_3+\delta, L_4/\sigma_3+\delta) \geq 1/4,
	    \end{equation}
	    on a fine enough mesh of $a_1,a_2,a_3$ (which induce an upper bound on $a_4$).
	    Notice the other subcases in the proof handle cases in which $a_4 \geq 1-a_1-a_3$ and 
	    \begin{equation}\label{eq:a4cond}
	    L_2-L_1 \leq 0.35\sqrt{1-a_1^2-a_2^2-2a_3^2}.
	    \end{equation}
	    
	    All expressions $L_i/\sigma_3$ and $a_3/\sigma_3$ and $(1-a_1-a_3)/\sigma_3$ have partial derivatives $< 10$, hence considering a mesh of $\set{(a_1,a_2,a_3)}{a_3\leq a_2\leq a_1 \leq 0.7, a_1+a_2+a_3 \geq 1, a_1+a_2\leq 1}$, with granularity $\delta / 15$ in every axis, we may verify~\eqref{eq:0.25} by checking~\eqref{eq:0.25pp} on the mesh points. One detail is that on the mesh points we bound $a_4$ by $1-a_1-a_3$ (instead of $\min(a_3, \sigma_3)$) only if $L_2-L_1+\delta/2 < 0.35\sqrt{1-a_1^2-a_2^2-2a_3^2}$, ensuring that if~\eqref{eq:a4cond} is not satisfied for a point, then its nearest mesh point will not use the improved bound $a_4 \leq 1-a_1-a_3$ (introducing `discontinuity'); this behavior is overridden to the points $(a_1, a_2, a_3) = (0.5\pm 0.02, 0.5\pm 0.02, 0.5\pm 0.02)$, since there~\eqref{eq:a4cond} is always satisfied. Choosing $\delta = 0.03$,~\eqref{eq:0.25pp} can be verified \cite{us} to all the described mesh points. 
	    
	\subsection{Proof of~\tops{$\eqref{eq:16}$}}\label{app:16}
	    Instead of checking~\eqref{eq:16}, we will check that
	    \begin{equation}\label{eq:16pp}
	        \be_{\varepsilon \in \spm^{2}}\left[ D\left( \frac{\min(a_2, 1-a_1-a_2)}{\sigma_2} + \delta, \frac{1+a_1 \varepsilon_1 + a_2 \varepsilon_2}{\sigma_2}  + \delta \right) \right] \geq 1/12
	    \end{equation}
 	    with $\sigma_2 = \sqrt{1-a_1^2-a_2^2}$ is satisfied on a mesh of points in $\set{(a_1, a_2)}{a_1+a_2 \leq 1, a_2 \leq a_1 \in [0.4, 0.6]}$. Since all the involved arguments fed to $D$ are $10$-Lipschitz, it suffices we verify~\eqref{eq:16pp} on a mesh with $\delta/10$ granularity in every axis. Verification \cite{us} can be done with $\delta=0.01$.
	    
	    \section{Family \tops{$F_2(\delta_0)$}}\label{fam2}
	    
	    In this appendix, we prove the following result, which together with Proposition~\ref{firstpart}, implies Proposition~\ref{combinedprops}.
	    
	    \begin{proposition}\label{secondpart}
For $\delta_0=10^{-9}$, we have $\pr \big[ X \geq 1 \big] \geq \frac{7}{64}$ for all collections $\lbrace a_i \rbrace$ in $ F_2(\delta_0)$.    
\end{proposition}
	    
	    To prove Proposition~\ref{secondpart}, take smallest possible $\delta>0$ such that $a_1 \in \left[ \frac{2}{3}-\delta,\frac{2}{3}+\delta \right]$ and $a_2 \in \left[ \frac{1}{3}-\delta,\frac{1}{3}+\delta \right]$. Assume $\delta \leq 10^{-9}$. Assume our collection $\lbrace a_1,\ldots,a_n \rbrace$ has $\pr \big[ |X| \geq 1 \big]<\frac{7}{32}$. We will derive a contradiction.

Note that 

\begin{equation}\label{triv23}
1-a_1-a_2 \leq 2 \delta.    
\end{equation}

\noindent
Denote
\[
p_1=\pr \big[|\sum_{i=3}^{n} a_i \varepsilon_{i}|\geq 1-a_1-a_2 \big], \quad p_2=\pr  \big[|\sum_{i=3}^{n} a_i \varepsilon_{i}|\geq 1-a_1+a_2 \big], \quad p_3=\pr \big[|\sum_{i=3}^{n} a_i \varepsilon_{i}|\geq 1+a_1-a_2 \big].
\]
\noindent
Note that $\pr \big[ |X| \geq 1 \big] \geq \frac{1}{4}(p_1+p_2+p_3)$, so it is enough to show that

\begin{equation}\label{goalf2}
p_1+p_2+p_3 \geq \frac{7}{8}.    
\end{equation}

\noindent
We can assume 

\begin{equation}\label{veryeasy}
a_3+a_4+a_5<1,    
\end{equation}
else we would be done by Observation \ref{keyantichain}. We will make consecutive claims about the collection $\lbrace a_1,\ldots,a_n \rbrace$, characterizing it more and more precisely until we are ready to obtain a contradiction.

Note that for $\eta=10^{-5}$, the following two lemmas hold. 

\begin{lemma}\label{ess1}
Assume $b_1 \geq \ldots \geq b_m>0$, $\sum_{i=1}^{m}b_i^{2}=1$ and $b_1 \leq \frac{1}{2}+\eta$. Then $$\pr \big[|\sum_{i=1}^{m} b_i \varepsilon_{i}|\geq 4 \delta\big] \geq \frac{5}{8}.$$
\end{lemma}

\begin{proof}
Note that if $b_3 \geq 4 \delta$, we are done by Observation \ref{keyantichain}. So we only need to consider the case when $b_3<4 \delta$. First, we argue that we have

\begin{equation}\label{condnneeded}
\sum_{i=3}^{m}b_i^{2} \geq 960 \delta^{2}.    
\end{equation}

\noindent
Since we know that $$\sum_{i=3}^{m}b_i^{2}  \geq 1-2(\frac{1}{2}+\eta)^{2}=\frac{1}{2}-2\eta-2\eta^{2},$$ to prove \eqref{condnneeded} holds, it is enough to show

\begin{equation}\label{needtoensure}
960 \delta^{2} \leq \frac{1}{2}-2\eta-2\eta^{2}.
\end{equation}

But \eqref{needtoensure} trivially holds as $\delta \leq 10^{-9}, \eta=10^{-5}$.
	    
	    Now using $4\delta>b_3,\ldots,b_m>0$ and \eqref{condnneeded}, we know that we can choose $10$ disjoint subsets $S_1,\ldots,S_{10}$ of $\lbrace b_3,\ldots,b_m \rbrace$ such that for $1 \leq i \leq 10$, we have
	    $$96 \delta^{2} \geq \sum_{b_j \in S_i} b_j^{2} \geq 80 \delta^{2}.$$
	    Then for each of these sets $S_i$, we consider the random process $W(S_i;4 \delta)$. By Lemma \ref{hittingprob}, each of these is successful with probability at least $\frac{1}{2}$ and independently of the other ones. If for some $t$, $1 \leq t \leq 10$, we condition on the event $E_t$ that precisely $t$ of these processes are successful, Observation \ref{keyantichain} ensures that
	    \[
	    \pr\big[|\sum_{i=1}^{m} b_i \varepsilon_i |< 4 \delta | E_t\big] \leq \binom{t}{\lfloor t/2 \rfloor} {2^{-t}}.
	    \]
	    So we can bound $$\pr\big[|\sum_{i=1}^{m} b_i \varepsilon_i |< 4 \delta \big] \leq  2^{-10}+ 2^{-10}\sum_{t=1}^{10} \binom{10}{t} \binom{t}{\lfloor t/2 \rfloor}  {2^{-t}} \leq \frac{3}{8}.$$ This now finishes the proof of Lemma \ref{ess1}.    
\end{proof}

\begin{lemma}\label{ess2}
Assume $b_1 \geq \ldots \geq b_m>0$, $\sum_{i=1}^{m}b_i^{2}=1$ and $b_1 \leq \frac{1}{2}+\eta$. Then $$\pr\big[|\sum_{i=1}^{m} b_i \varepsilon_{i}|\geq 1+4  \delta \big] \geq \frac{1}{8}.$$
\end{lemma}

\begin{proof}
This follows directly using \eqref{eq:stash-of-D} by $D(0.51,1.01)=\frac{1}{16}$. 
\end{proof}

Now we can use these lemmas to prove the following corollary.

\begin{corollary}\label{startfrom}
We have $p_1 \geq \frac{5}{8}$ and $p_2 \geq \frac{1}{8}$.
\end{corollary}

\begin{proof}
Note that $a_3 \leq a_2 \leq \frac{1}{3}+\delta$ and $$\sum_{i=3}^{n}a_i^{2} \geq 1-(\frac{2}{3}+\delta)^{2}-(\frac{1}{3}-\delta)^{2}=\frac{4}{9}-\frac{2}{3}\delta-2 \delta^{2}.$$ So for $\delta \leq 10^{-9},\eta=10^{-5}$, we see that we have $$a_3 \leq (\frac{1}{2}+\eta)\sqrt{\sum_{i=3}^{n}a_i^{2}}.$$ 

Thus we conclude from Lemma \ref{ess1} that

\begin{align*}
p_1&=\pr\big[|\sum_{i=3}^{n} a_i \varepsilon_{i}|\geq 1-a_1-a_2\big] \\
&\geq \pr\big[|\sum_{i=3}^{n} a_i \varepsilon_{i}|\geq 2 \delta\big] \\
&\geq \pr\big[|\sum_{i=3}^{n} a_i \varepsilon_{i}|\geq 4 \delta \sqrt{\sum_{i=3}^{n}a_i^{2}}\big] \\
&\geq \frac{5}{8}.
\end{align*}

Analogously, we conclude from Lemma \ref{ess2} that 
\begin{align*}
p_2 &=\pr\big[|\sum_{i=3}^{n} a_i \varepsilon_{i}|\geq 1-a_1+a_2\big] \\
&\geq \pr\big[|\sum_{i=3}^{n} a_i \varepsilon_{i}|\geq \frac{2}{3}+2 \delta\big] \\
&\geq \pr\big[|\sum_{i=3}^{n} a_i \varepsilon_{i}|\geq (1+4 \delta) \sqrt{\sum_{i=3}^{n}a_i^{2}}\big] \\
&\geq \frac{1}{8}.
\end{align*}
\end{proof}



\begin{claim}
We have $a_3 > 14 \delta$.
\end{claim}

\begin{proof}
Assume we had $a_3 \leq 14 \delta$. By our choice of $\delta$, we can trivially check that 

\begin{equation}\label{wantfordiv}
\sum_{i=3}^{n}a_i^{2} \geq 1-(\frac{2}{3}+\delta)^{2}-(\frac{1}{3}-\delta)^{2} \geq 3920 \delta^{2}.    
\end{equation}

Using \eqref{wantfordiv}, we can choose $20$ disjoint subsets (possibly containing a single element) $S_1,\ldots,S_{20}$ of $\lbrace a_3,\ldots,a_n \rbrace$ such that for $1 \leq i \leq 20$, either $S_i$ contains a single element $a_i \geq 2 \delta$, or all its elements are smaller than $2 \delta$ and we have
\[
24 \delta^{2} \geq \sum_{b_j \in S_i} b_j^{2} \geq 20 \delta^{2}.
\]
Then for each of these sets $S_i$, consider the random process $W(S_i;2\delta)$. By Lemma \ref{hittingprob}, each of these is successful with probability at least $\frac{1}{2}$ and independently of the other ones. If for some $t$, $1 \leq t \leq 20$, we condition on the event $F_t$ that precisely $t$ of these processes are successful, Observation \ref{keyantichain} ensures that
\[
\pr\big[|\sum_{i=3}^{n} a_i \varepsilon_i |< 2 \delta | F_t\big] \leq \binom{t}{\lfloor t/2 \rfloor} {2^{-t}}.
\]
So we can bound 
\[
1-p_1 \leq \pr\big[|\sum_{i=3}^{n}a_i|< 2 \delta \big] \leq 2^{-20}+2^{-20}\sum_{t=1}^{20} \binom{20}{t} \binom{t}{\lfloor t/2 \rfloor} {2^{-t}} \leq \frac{1}{4}.
\]
Combining $p_1 \geq \frac{3}{4}$ with $p_2 \geq \frac{1}{8}$ that we have proven before, this verifies \eqref{goalf2}.
\end{proof}

Let $k$ be an integer such that $a_{k-1} \geq 1-a_1-a_2$, but $a_{k}<1-a_1-a_2$ (if $a_n \geq 1-a_1-a_2$, set $k=n+1$).

\begin{claim}\label{toomanysmall}
Let $k$ be the smallest integer such that $a_k<1-a_1-a_2$. Then we have $\sum_{i=k}^{n}a_i^{2} < 328 \delta^{2}$.
\end{claim}

\begin{proof}
Assume that we had $\sum_{i=k}^{n}a_i^{2} \geq 328 \delta^{2}$. Then we can find disjoint subsets $S_1,S_2,T_1,\ldots,T_5$ of $\lbrace a_k,\ldots,a_n \rbrace$ such that the following holds. For $x=1,2$ we have $$104 \delta^{2} \geq \sum_{i \in S_x}a_i^{2} \geq 100 \delta^{2}$$ and for $y=1,\ldots,5$, we have  $$24 \delta^{2} \geq \sum_{i \in T_y}a_i^{2} \geq 20 \delta^{2}.$$ Now consider the random processes $$W(S_1;6\delta), W(S_2;6\delta),W(T_1;2\delta),\ldots,W(T_5;2\delta).$$ By Lemma \ref{hittingprob}, each of these is successful with probability at least $\frac{1}{2}$ and independently of the other ones. We apply Observations \ref{k2} and \ref{k3}, using $a_3$ and $r(S_1;6\delta), \ldots,r(T_5;2\delta)$, to bound $p_1$. With probability at least $\frac{93}{128}$, both some process corresponding to $S_x$ and some process corresponding to $T_y$ are successful, and conditional on that we get the lower bound of $\frac{7}{8}$ on $\pr\big[|Y|\geq 1-a_1-a_2\big]$. Further, we get the lower bound of $\frac{3}{4}$ on $\pr\big[|Y|\geq 1-a_1-a_2\big]$ if either some process corresponding to $S_x$ or some process corresponding to $T_y$ are successful, and the lower bound of $\frac{1}{2}$ otherwise (this last case happens at most with probability $\frac{1}{128}$). So overall, we obtain $p_1>\frac{3}{4}$. Combining that with $p_2 \geq \frac{1}{8}$ that we have proven in Corollary \ref{startfrom}, we verify that \eqref{goalf2} holds. 
\end{proof}

\begin{claim}\label{veryclose}
Let $k$ be the smallest integer such that $a_k<1-a_1-a_2$. Then we have $a_3-a_{k-1}<1-a_1-a_2 < 2 \delta$.
\end{claim}

\begin{proof}
If we had two terms $a_s,a_t>1-a_1-a_2$ such that $|a_s-a_t| \geq 1-a_1-a_2$, Observation \ref{k2} for $a_s,a_t$ gives $p_1 \geq \frac{3}{4}$. Combining that with $p_2 \geq \frac{1}{8}$ verifies \eqref{goalf2}.        
\end{proof}

\begin{claim}
Let $k$ be the smallest integer such that $a_k<1-a_1-a_2$. Then we have $k<12$.
\end{claim}

\begin{proof}
If we had $k \geq 12$, by Observation \ref{keyantichain}, we have $p_1 \geq \frac{193}{256}>\frac{3}{4}$, and combining that with $p_2 \geq \frac{1}{8}$ verifies \eqref{goalf2}.        
\end{proof}

\begin{claim}
Let $k$ be the smallest integer such that $a_k<1-a_1-a_2$. Then we have $k<8$.
\end{claim}

\begin{proof}
We have already shown that $k \leq 11$. If $11 \geq k \geq 8$, note that using our choice of $\delta$, Claim \ref{toomanysmall} and Claim \ref{veryclose}, we get $$a_5+a_6+a_7 \geq \frac{2}{3}+2 \delta \geq 1-a_1+a_2.$$ That gives $p_2 \geq \frac{7}{32}$.

Since $k \geq 8$, we also have $p_1 \geq \frac{11}{16}$ by Observation \ref{keyantichain}. Hence we verify \eqref{goalf2}.
\end{proof}

\begin{claim}
Let $k$ be the smallest integer such that $a_k<1-a_1-a_2$. Then we have $k \geq 7$ (and hence as also $k \leq 7$, we have $k=7$).
\end{claim}

\begin{proof}
By our choice of $\delta$, we have $\sum_{i=6}^{n}a_i^{2} \geq 328 \delta^{2}$, and result thus follows by Claim \ref{toomanysmall}.        
\end{proof}

We will now show that $\sum_{i=7}^{n}a_i^{2} \geq 328 \delta^{2}$ (which together with Claim \ref{toomanysmall} gives a desired contradiction). By our definition of $\delta$ and assumption that $a_1+a_2<1$, we either have $a_1=\frac{2}{3}-\delta$ or $a_2=\frac{1}{3}-\delta$.

First consider the case $a_2=\frac{1}{3}-\delta$. Then $$\sum_{i=1}^{6}a_i^{2} \leq (\frac{2}{3}+\delta)^{2}+5(\frac{1}{3}-\delta)^{2}= 1-2\delta+6\delta^{2},$$ and hence $$\sum_{i=7}^{n}a_i^{2} \geq 2\delta-6\delta^{2} > 328 \delta^{2}$$ for every $0<\delta<\frac{1}{167}$. 

So we can assume that instead $a_1=\frac{2}{3}-\delta$. But now we use \eqref{veryeasy} to bound $$\sum_{i=1}^{6}a_i^{2} \leq (\frac{2}{3}-\delta)^{2}+(\frac{1}{3}+\delta)^{2}+4(\frac{1}{3})^{2} \leq 1 - \frac{2}{3}\delta+2\delta^{2},$$ and hence $$\sum_{i=7}^{n}a_i^{2} \geq  \frac{2}{3}\delta+2\delta^{2} >328 \delta^{2}$$ for every $0<\delta<\frac{1}{489}$. 

Thus we reached a desired contradiction, and the proof of Proposition \ref{secondpart} is complete. $\square$

	\end{appendices}
	
\end{document}